\documentclass[10pt]{imsart}
\usepackage{fullpage}
\RequirePackage[OT1]{fontenc}
\RequirePackage{amsthm,amsmath,amssymb,natbib}
\RequirePackage[colorlinks,citecolor=blue,urlcolor=blue]{hyperref}
\RequirePackage{graphicx}
\RequirePackage{enumerate}
\RequirePackage{MnSymbol}
\RequirePackage{xcolor}
\RequirePackage{subfig}
\RequirePackage{parskip}
\RequirePackage{url}
\RequirePackage{bm}
\RequirePackage{booktabs}

\arxiv{arXiv:1409.2344}

\startlocaldefs
\numberwithin{equation}{section}
\newtheorem{theorem}{Theorem}
\newtheorem{lemma}[theorem]{Lemma}

\newtheorem{proposition}[theorem]{Proposition}
\theoremstyle{remark}
\newtheorem*{remark}{Remark}
\newtheorem{assumption}{Assumption}
\newtheorem{definition}{Definition}

\def\clap#1{\hbox to 0pt{\hss#1\hss}}

\newcommand{\Ex}{\mathbb{E}}

\endlocaldefs

\begin{document}

\bibliographystyle{plainnat}

\begin{frontmatter}
  \title{A nonparametric two-sample hypothesis testing problem for random graphs}
  \runtitle{A nonparametric two-sample hypothesis testing problem for random graphs}

\begin{aug}
\author{\fnms{Minh} \snm{Tang}\thanksref{a}\ead[label=e1]{mtang10@jhu.edu}},
\author{\fnms{Avanti} \snm{Athreya}\thanksref{a}\ead[label=e2]{dathrey1@jhu.edu}},
\author{\fnms{Daniel L.} \snm{Sussman}\thanksref{b}%
\ead[label=e3]{daniellsussman@fas.harvard.edu}}, \\
\author{\fnms{Vince} \snm{Lyzinski}\thanksref{c}%
\ead[label=e4]{vlyzins1@jhu.edu}}
\and
\author{\fnms{Carey E.} \snm{Priebe}\thanksref{a}%
\ead[label=e5]{cep@jhu.edu}}

\address[a]{Department of Applied Mathematics and Statistics, Johns
  Hopkins University \\
\printead{e1,e2,e5}}

\address[b]{Department of Statistics, Harvard University \\ \printead{e3}}

\address[c]{Human Language Technology Center of Excellence, Johns
  Hopkins University \\
\printead{e4}}

\runauthor{M.~Tang et al.}

\affiliation{Johns Hopkins University and Harvard University}

\end{aug}

\begin{abstract}
  We consider the problem of testing whether two independent finite-dimensional
  random dot product graphs have generating latent positions that are
  drawn from the same distribution, or distributions
  that are related via scaling or projection. We propose a test
  statistic that is a kernel-based function of the estimated latent
  positions obtained from the adjacency spectral embedding for each graph. 
  We show that our test statistic using the estimated latent positions
  converges to the test statistic obtained using the true but unknown
  latent positions and hence that our proposed test procedure is
  consistent across a broad range of alternatives. Our proof of 
  consistency hinges upon a novel concentration inequality
  for the suprema of an empirical process in the estimated latent
  positions setting. 
\end{abstract}


\end{frontmatter}

\section{Introduction}
\label{sec:introduction}
The nonparametric two-sample hypothesis testing problem
involves
\begin{gather*}
  \{X_i\}_{i=1}^{n} \overset{\mathrm{i.i.d}}{\sim} F, \quad    
  \{Y_k\}_{k=1}^{m} \overset{\mathrm{i.i.d}}{\sim} G; \quad
  \mathbb{H}_0 \colon F = G \quad \text{against} \quad \mathbb{H}_{A} \colon
F \not = G 
\end{gather*}
where $F$ and $G$ are two distributions taking values in
$\mathbb{R}^{d}$. 
This is a classical problem and there exist a large number of test
statistics $T(\{X_i\}_{i=1}^{n}, \{Y_k\}_{k=1}^{m})$ that are
consistent for any arbitrary distributions $F$ and $G$.

In this paper, we consider a related problem that arises naturally in
the context of inference on random graphs. That is, suppose that the
$\{X_i\}_{i=1}^{n}$ and $\{Y_k\}_{k=1}^{m}$ are {\em unobserved}, and we
observe instead adjacency matrices $\mathbf{A}$ and $\mathbf{B}$
corresponding to random dot product graphs on $n$ and $m$ vertices
with latent positions $\{X_i\}_{i=1}^{n}$ and
$\{Y_k\}_{k=1}^{m}$, respectively. Denoting by $\{\hat{X}_i\}_{i=1}^{n}$ and
$\{\hat{Y}_k\}_{k=1}^{m}$ the adjacency spectral embedding of $\mathbf{A}$ and
$\mathbf{B}$ (see Definition~\ref{def:ase}), we construct test statistics
$T(\{\hat{X}\}_{i=1}^{n}, \{\hat{Y}_k\}_{k=1}^{m})$ for testing $F = G$
(and related hypotheses) that are consistent for a broad collection of
distributions.

In other words, we construct a test for the hypothesis that two random
dot product graphs have the same underlying distribution of latent
positions, or underlying distributions that are related via scaling or
projection. This problem may be viewed as the nonparametric analogue
of the semiparametric inference problem considered in
\citet{tang14:_two}, in which a valid test is given for the hypothesis
that two random dot product graphs have the same fixed latent
positions.  This formulation also includes, as a special case, a
test for the parametric problem of whether two graphs come from the
same stochastic blockmodel (where the block probability matrix is
positive semidefinite) or from the same degree-corrected stochastic
blockmodel.  Determining whether two random graphs are ``similar" in
an appropriate sense is a problem that arises naturally in
neuroscience, network analysis, and machine learning. 
Examples include the comparison of graphs in a time series,
such as email correspondence among a group over time, the comparison
of neuroimaging scans of patients under varying conditions, or the comparison of user
behavior on different social media platforms. 

While it might seem like there are only minor differences between the
nonparametric setting of the current paper and the semiparametric
setting of \cite{tang14:_two}, the implications with regard to
inference are quite significant. Indeed, in the semiparametric
setting, the graphs are on the same vertex set with known vertex
alignment; in the nonparametric setting we consider herein, the graphs need not be on the
same vertex set or even have the same number of vertices. This
difference implies that the nonparametric testing procedure of the
current paper is applicable in more general and diverse settings; on
the other hand, when the vertex correspondences exist and are known,
the semiparametric testing procedure has more power.  Secondly, in the
semiparametric setting, the dimensionality of the hypotheses (the
number of parameters) increases with $n$, the number of vertices, while
in the current setup the hypotheses are fixed for all $n$. As
such, the notion of a consistent test procedure in \citep{tang14:_two}
is considerably more subtle. Finally, while rejection regions can be
theoretically derived for the test procedures in both the
nonparametric setting and the semiparametric setting, in practice they
are usually estimated via some bootstrap resampling procedure. For the
nonparametric setting wherein the null hypothesis is fixed as the size
of the graphs changes, bootstrap resampling is straightforward. A
feasible bootstrapping procedure in the semiparametric setting 
is much more involved.


The test statistic we construct is an empirical estimate of the
maximum mean discrepancy of \citet{gretton12:_kernel_two_sampl_test}. 
The maximum mean discrepancy in this context is
equivalent to an $L_2$-distance between kernel density estimates of
distributions of the latent positions (see
e.g. \citet{anderson94:_two}). The test statistic can also be framed
as a weighted $L_2$-distance between empirical estimates of
characteristic functions similar to those of
\citet{hall13,fernandez08,baringhaus88}. 
Indeed, techniques for the estimation and comparison of densities or
characteristic functions given i.i.d data are well-known.  We strongly
emphasize, however, that in our case, the observed data are {\em not
  the true latent positions}---which are themselves random and drawn
from the unknown distributions whose equality we wish to test---but
rather the adjacency matrices of the resulting random dot product
graphs.  Thus one of our main technical contributions is the
demonstration that functions of the true latent positions are
well-approximated by functions of the adjacency spectral embeddings.

The results of this paper are mainly for dense graphs, i.e., those graphs for
which the average degree scale linearly with the number of 
vertices. Analogous results for non-dense graphs, e.g., those for which the
average degree of the vertices grows at order $\Omega(\log^{4} n)$ --
$n$ being the number of vertices in the graph -- are more subtle and we
touch upon this briefly in Section~\ref{sec:extensions}. 



We organize the paper as follows.  In
Section~\ref{sec:setting_background}, we recall the definition of a
random dot product graph and the adjacency spectral embedding; we
review the relevant background in kernel-based hypothesis testing; and
we formulate a nonparametric two-sample test of equality of
distributions for the latent positions of a pair of random dot product
graphs. In Section~\ref{sec:main-results}, we propose a test procedure
for the two-sample test of equality up to orthogonal transformation in
which the test statistics are a function of the adjacency spectral
embedding. We note that our hypotheses of
equality are purely a function of the non-identifiability of the
random dot product graph model. This non-identifiability also restricts our
consideration of kernel-based hypothesis testing to radial kernels.  
We establish the consistency of our test procedure by
deriving a novel concentration inequality for the suprema of an
empirical process using the estimated latent positions.  In
Section~\ref{sec:experimental-results}, we illustrate our test
procedure with experimental results on simulated and real
data. Section~\ref{sec:extensions} extends the test procedure in
Section~\ref{sec:main-results} to consider looser notions of equality
between the two distributions as well as sparsity in the underlying
graphs model. 
\section{Background and Setting}
\label{sec:setting_background}
We first recall the notion of a random dot product graph
\citep{young2007random}.
\begin{definition}
  \label{def:1}
  Let $\Omega$ be a subset of $\mathbb{R}^{d}$ such that, for all
  $\omega_1, \omega_2 \in \Omega$, the inner product $\langle
  \omega_1, \omega_2 \rangle = \omega_1^{\top}
  \omega_2$ is contained in the interval $[0,1]$. For any given $n
  \geq 1$, let $\mathbf{X} = [X_1, X_2 \dots, X_n]^{\top}$ be a $n \times
  d$ matrix whose rows are arbitrary elements of $\Omega$. Given ${\bf
    X}$, suppose $\mathbf{A}$ is a random $n \times n$ adjacency
  matrix with probability
  \begin{equation*} \mathbb{P}[\mathbf{A} | \{X_i\}_{i=1}^{n}] =
    \prod_{i < j} (X_i^{\top} X_j)^{\mathbf{A}_{ij}} (1 - X_i^{\top}
    X_j)^{1 - \mathbf{A}_{ij}}. 
  \end{equation*} 
  $\mathbf{A}$ is then said to be the adjacency matrix of a {\em random dot product graph}
  (RDPG) with {\em latent positions} $\mathbf{X}$ and we denote this by $\mathbf{A} \sim
  \mathrm{RDPG}(\mathbf{X})$. Now suppose that the rows of
  $\mathbf{X}$ are not fixed, but are instead independent random variables sampled
  according to some distribution $F$ on $\Omega$. Then $\mathbf{A}$ is
  said to be the adjacency matrix of a {\em random dot product graph}
  with {\em latent positions} $\mathbf{X}$ {\em sampled according to}
  $F$ and we denote this by writing $(\mathbf{X}, \mathbf{A}) \sim
  \mathrm{RDPG}(F)$. We shall also write $\mathbf{A} \sim
  \mathrm{RDPG}(F)$ when the dependency of $\mathbf{A}$ on $\mathbf{X}$
  is integrated out.
\end{definition}

As an example of random dot product graphs, one could take $\Omega$ to
be the unit simplex in $\mathbb{R}^{d}$ and let $F$ be a mixture of
Dirichlet distributions. Given a matrix of latent positions ${\bf X}$, the random dot product
model generates a symmetric adjacency matrix ${\bf A}$ whose edges
$\{\mathbf{A}_{ij}\}_{i < j}$ are independent Bernoulli random
variables with parameters $\{\mathbf{P}_{ij}\}_{i < j}$, where
$\mathbf{P} = {\bf XX}^T$. Random dot product graphs are a specific
example of {\em latent position graphs} \citep{Hoff2002}, in which
each vertex is associated with a latent position and, conditioned on
the latent positions, the presence or absence of the edges in the
graph are independent. The edge presence probability between two
vertices is given by a symmetric link function of the latent positions
of the associated vertices. A random dot product graph with i.i.d latent positions on $n$ vertices
is also, when viewed as an induced subgraph of an infinite graph, an
example of an {\em exchangeable random graph}
\citep{diaconis08:_graph_limit_exchan_random_graph}. Random dot product
graphs are related to stochastic block model graphs
\citep{Holland1983} and degree-corrected stochastic block model graphs
\citep{karrer2011stochastic}, as well as mixed membership block
models \citep{Airoldi2008}; for example, a stochastic block model graph
with $K$ blocks and a positive semidefinite block probability matrix
$\mathbf{B}$ corresponds to a random dot product graph whose latent positions are
drawn from a mixture of $K$ point masses. 

\begin{remark}
  We note that non-identifiability is a property of nearly all exchangeable random graph models,
  and specifically, it is an intrinsic property of random  dot product graphs.
  Indeed, for any matrix $\mathbf{X}$ and
  any orthogonal matrix $\mathbf{W}$, the inner product between any
  rows $i,j$ of $\mathbf{X}$ is identical to that between the rows
  $i,j$ of $\mathbf{XW}$. Hence, for
  any probability distribution $F$ on $\Omega$ and unitary operator
  $U$, the adjacency matrices $\mathbf{A}
  \sim \mathrm{RDPG}(F)$ and $\mathbf{B} \sim \mathrm{RDPG}(F \circ
  U)$ are identically distributed (here, for a random variable $X
  \sim F$, we write $F \circ U$ to denote the distribution of $Y =
  U^{\top} X$).
\end{remark}

We now define the notion of adjacency spectral embedding; this is the
key intermediate step in our subsequent two-sample hypothesis testing
procedures. 
\begin{definition} \label{def:ase} Let $\mathbf{A}$ be a $n \times n$
  adjacency matrix. Suppose the eigendecomposition of
  $|\mathbf{A}| = (\mathbf{A}^{\top} \mathbf{A})^{1/2}$ is given
  by
  $$|\mathbf{A}| = \sum_{i=1}^{n} \lambda_i \bm{u}_i \bm{u}_i^{\top}$$
  with $\lambda_1 \geq \lambda_2 \geq \dots \geq \lambda_n$ being the
  eigenvalues of $|\mathbf{A}|$ and $\bm{u}_1, \dots, \bm{u}_n$ the
  corresponding eigenvectors. Given a positive integer $d \leq n$,
  denote by
  $\mathbf{S}_{\mathbf{A}} = \mathrm{diag}(\lambda_1, \dots,
  \lambda_d)$
  the diagonal matrix whose diagonal entries are
  $\lambda_1, \dots, \lambda_d$ and denote by
  $\mathbf{U}_{\mathbf{A}}$ the $n \times d$ matrix whose columns are
  the corresponding eigenvectors $\bm{u}_1, \dots, \bm{u}_d$.  The
  {\em adjacency spectral embedding} $\mathbf{A}$ into
  $\mathbb{R}^{d}$ is then the $n \times d$ matrix
  $\hat{{\bf X}} = {\bf U}_{\mathbf{A}} {\bf S}_{\mathbf{A}}^{1/2}$.
\end{definition}
\begin{remark}
  The intuition behind the notion of adjacency spectral embedding is
  as follows. We note that if
  $(\mathbf{A}, \mathbf{X}) \sim \mathrm{RDPG}(F)$, then the upper
  triangular entries of $\mathbf{A} - \mathbf{X} \mathbf{X}^{\top}$
  are independent random variables. Let $\|\cdot \|$ denote the
  spectral norm of a matrix. Then one can show that
  $\|\mathbf{A} - \mathbf{X} \mathbf{X}^{\top} \| = O(\|\mathbf{X}\|)
  = o(\|\mathbf{X} \mathbf{X}^{\top}\|)$
  with high probability \citep{oliveira2009concentration}.
  That is to say, $\mathbf{A}$ can be viewed as a ``small''
  perturbation of $\mathbf{X} \mathbf{X}^{\top}$.  If we now assume
  that $\mathbf{X}$ is of rank $d$ for some $d$ -- an assumption that
  is justified in the random dot product graphs model -- then the
  Davis-Kahan theorem \citep{davis70} implies that the subspace
  spanned by the top $d$ eigenvectors of $\mathbf{X}
  \mathbf{X}^{\top}$ is well-approximated by the subspace spanned by
  the top $d$ eigenvectors of $\mathbf{A}$. 
  In particular, the eigendecomposition of $\mathbf{X}
  \mathbf{X}^{\top}$ recovers
  the matrix $\mathbf{X}$ up to an orthogonal transformation; hence the adjacency
  spectral embedding of $\mathbf{A}$ is expected to yield a consistent
  estimate of $\mathbf{X}$ up to an orthogonal transformation (see 
  Lemma~\ref{lem:2}). 
\end{remark}

\subsection{Two-sample hypothesis testing }
\label{sec:hypothesis-tests}
In this paper we propose a nonparametric version of the two-sample
hypothesis test examined in \citet{tang14:_two}.  To wit,
\citet{tang14:_two} presents a two-sample random dot product graph
hypothesis test as follows.  Let ${\bf X}_n$ and ${\bf Y}_n$ be $n
\times d$ matrices of fixed (non-random) latent positions, and
$\mathcal{O}(d)$ the collection of orthogonal matrices in
$\mathbb{R}^{d \times d}$. Suppose ${\bf A} \sim
\mathrm{RDPG}(\mathbf{X}_n)$ and $\mathbf{B} \sim
\mathrm{RDPG}(\mathbf{Y}_n)$ are the adjacency matrices of random dot
product graphs with latent positions ${\bf X}_n$ and ${\bf Y}_n$,
respectively.  Consider the sequence of hypothesis tests
\begin{align*} H^n_0 \colon {\bf X}_n \upVdash {\bf Y}_n \quad 
\text{against} \quad H^n_A \colon {\bf X}_n \nupVdash {\bf Y}_n
\end{align*}
where $\upVdash$ denotes that there exists an ${\bf W} \in
\mathcal{O}(d)$ such that ${\bf X}_n={\bf Y}_n{\bf W}$.  In
\citet{tang14:_two}, it is shown that rejecting for large values of the
test statistic $T_n$ defined by
\begin{equation*}
T_n=\min\limits_{{\bf W} \in \mathcal{O}(d)}
\|\hat{{\bf X}}_n{\bf W}-\hat{{\bf Y}}_n\|_F,
\end{equation*}
yields a consistent test procedure for any sequence of latent positions
$\{\mathbf{X}_n\}$, $\{\mathbf{Y}_n\}$ for which $\min_{\mathbf{W} \in
  \mathcal{O}(d)} \|\mathbf{X}_n - \mathbf{Y}_n \mathbf{W}\|$ diverges
as $n \rightarrow \infty$.

Our main point of departure in this work is the assumption
that, for each $n$, the rows of the latent positions $\mathbf{X}_n$ and
$\mathbf{Y}_n$ are independent samples from some fixed distributions $F$ and $G$,
respectively. The corresponding tests are therefore tests of equality between $F$ and
$G$. More formally, we consider the following two-sample nonparametric testing
problems for random dot product graphs. Let $F$ and $G$ be probability
distributions on $\Omega \subset \mathbb{R}^{d}$ for some $d$. 
Given $\mathbf{A} \sim \mathrm{RDPG}(F)$ and $\mathbf{B} \sim \mathrm{RDPG}(G)$, we
consider the tests:
\begin{enumerate}
\item{\em (Equality, up to orthogonal transformation)}
  \begin{align*}
  H_{0} \colon F \upVdash G
  \quad \text{against} \quad H_{A} \colon F  \nupVdash G,
  \end{align*}
  where $F \upVdash G$ denotes that there exists a unitary operator
  $U$ on $\mathbb{R}^{d}$ such that $F = G \circ U$ and $F \nupVdash
  G$ denotes that $F \not = G \circ U$ for any unitary operator $U$ on
  $\mathbb{R}^{d}$.
\item {\em (Equality, up to scaling)}
  \begin{align*}
H_{0} \colon F \upVdash G \circ c \quad \text{for
    some $c > 0$} \quad  
  \text{against} \quad H_{A} \colon F  \nupVdash G \circ c \quad
  \text{for any  $c > 0$},
  \end{align*}
  where $Y \sim F \circ c$ if $cY \sim F$.
\item {\em (Equality, up to projection)}   \begin{align*}
\quad H_{0} \colon F \circ \pi^{-1} \upVdash G
\circ \pi^{-1} \quad 
  \text{against} \quad H_{A} \colon F \circ \pi^{-1} \nupVdash G \circ \pi^{-1} ,
  \end{align*}
  where $\pi$ is the
  projection $x \mapsto x/\|x\|$; hence $Y \sim F \circ
  \pi^{-1}$ if $\pi^{-1}(Y) \sim F$. 
\end{enumerate}
We note that the above null hypothesis are nested; $F \upVdash G$
implies $F \upVdash G \circ c $ for $c = 1$ while $F \upVdash G
\circ c$ for some $c > 0$ implies $F \circ \pi^{-1} \upVdash G \circ
\pi^{-1}$.

\subsection{Maximum mean discrepancy}
\label{sec:kernel-based-two}
We now introduce the notion of the maximum mean discrepancy between
two distribution 
\citet{gretton12:_kernel_two_sampl_test}. 
The maximum mean discrepancy is a distance measure for probability
distributions and hence can be used to construct a non-parametric two-sample
hypothesis testing procedure (see
Theorem~\ref{thm:mmd_unbiased_limiting} below). The maximum mean
discrepancy is just one of several examples of kernel-based testing
procedures; see \citet{harchaoui13:_kernel} for a recent survey of the
literature and for a more detailed discussion.  

Let $\Omega$ be a
compact metric space and
$\kappa \, \colon\, \Omega \times \Omega \mapsto \mathbb{R}$ a
continuous, symmetric, and positive definite kernel on
$\Omega$. Denote by $\mathcal{H}$ the reproducing kernel Hilbert space
associated with $\kappa$. Now let $F$ be a probability distribution on
$\Omega$. Under mild conditions on $\kappa$, the map $\mu[F]$ defined
by
\begin{equation*}
  \mu[F] :=  \int_{\Omega} \kappa(\omega, \cdot) \, \mathrm{d} F(\omega)
\end{equation*}
belongs to $\mathcal{H}$.
Now, for given probability distributions $F$ and $G$ on $\Omega$, the
{\em maximum mean discrepancy} between $F$ and $G$ with respect to
$\mathcal{H}$ is the measure
\begin{equation*}
  \mathrm{MMD}(F, G; \mathcal{H}) := \|\mu[F] - \mu[G]
    \|_{\mathcal{H}}. 
\end{equation*}
We summarize some important properties of the maximum mean
discrepancy from \citet{gretton12:_kernel_two_sampl_test}. In
particular, if $\kappa$ is chosen so that $\mu$ is an
injective map, then $\|\mu[F] -
\mu[G]\|_{\mathcal{H}}$ yields a consistent test for testing the
hypothesis $\mathbb{H}_{0} \colon F = G$ against the hypothesis
$\mathbb{H}_{A} \colon F \not = G$ for any two arbitrary but fixed distributions $F$ and
$G$ on $\Omega$. 
\begin{theorem}
  \label{thm:mmd_unbiased_limiting}
  Let $\kappa
  \, \colon \,
  \mathcal{X} \times \mathcal{X} \mapsto \mathbb{R}$ be a positive definite
  kernel and denote by $\mathcal{H}$ the reproducing kernel Hilbert space
  associated with $\kappa$. Let $F$ and $G$ be probability distributions on $\Omega$; $X$
  and $X'$ independent random variables with distribution $F$, $Y$ and
  $Y'$ independent random variables with distribution $G$, and $X$ is
  independent of $Y$. Then
  \begin{equation}
    \label{eq:5}
    \begin{split}
    \| \mu[F] - \mu[G] \|^{2}_{\mathcal{H}} &= \sup_{h \in \mathcal{H}
    \colon \|h\|_{\mathcal{H}} \leq 1} |\mathbb{E}_{F}[h] -
    \mathbb{E}_{G}[h]|^{2} \\ &= \mathbb{E}[\kappa(X,X')] - 2 \mathbb{E}[\kappa(X,Y)]
    + \mathbb{E}[\kappa(Y,Y')].
    \end{split}
  \end{equation}
  Given $\mathbf{X} = \{X_i\}_{i=1}^{n}$ and $\mathbf{Y}
  = \{Y_k\}_{k=1}^{m}$ with $\{X_i\}
  \overset{\mathrm{i.i.d}}{\sim} F$ and $\{Y_i\}
  \overset{\mathrm{i.i.d}}{\sim} G$, 
  the quantity $U_{n,m}({\bf X}, {\bf Y})$ defined by
  \begin{equation}
    \label{eq:10}
    \begin{split}
     U_{n,m}({\bf X}, {\bf Y})
 &= \frac{1}{n(n-1)} 
     \sum_{j\not = i} \kappa(X_i,X_j)
    - \frac{2}{mn} \sum_{i=1}^{n} \sum_{k=1}^{m} \kappa(X_i, Y_k) \\ &+
    \frac{1}{m(m-1)} \sum_{l \not = k} \kappa(Y_k, Y_l)
    \end{split}
  \end{equation}
  is an {unbiased consistent estimate} of $\|\mu[F] -
  \mu[G]\|_{\mathcal{H}}^{2}$. Denote by $\tilde{\kappa}$ the kernel
  \begin{equation*}
    \begin{split}
    \tilde{\kappa}(x,y) &= \kappa(x,y) - \mathbb{E}_{z}
    \kappa(x, z) - \mathbb{E}_{z'} \kappa(z', y) +
    \mathbb{E}_{z,z'} \kappa(z,z')
     \end{split}
  \end{equation*}
  where the expectation is taken with respect to $z, z' \sim F$. 
Suppose that $\tfrac{m}{m+n} \rightarrow
  \rho \in (0,1)$ as $m, n \rightarrow \infty$. Then under the null
  hypothesis of $F = G$, 
  \begin{equation}
    \label{eq:mmd-X}
    (m+n) U_{n,m}(\mathbf{X}, \mathbf{Y})
    \overset{d}{\longrightarrow} \frac{1}{\rho(1 - \rho)} \sum_{l=1}^{\infty}
    \lambda_{l} (\chi^{2}_{1l} - 1)
  \end{equation}
  where $\{\chi^{2}_{1l}\}_{l=1}^\infty$ is a sequence of independent $\chi^{2}$
  random variables with one degree of freedom, and $\{\lambda_{l}\}$
  are the eigenvalues of the integral operator $\mathcal{I}_{F,
    \tilde{\kappa}}:\mathcal{H}\mapsto \mathcal{H}$ defined as
  \begin{equation*}
    I_{F, \tilde{\kappa}}(\phi)(x)=\int_{\Omega} \phi(y)\tilde{\kappa}(x,y) dF(y).
    \end{equation*}
    Finally, if $\kappa$ is a universal or
    characteristic kernel \citep{sriperumbudur11:_univer_charac_kernel_rkhs_embed_measur,
    steinwart01:_suppor_vector_machin}, then $\mu$
  is an injective map, i.e., $\mu[F] = \mu[G]$ if and only if $F = G$.
\end{theorem}

\begin{remark}
  A kernel $\kappa \colon \mathcal{X} \times \mathcal{X} \mapsto
  \mathbb{R}$ is universal if $\kappa$ is a continuous function of
  both its arguments and if the reproducing kernel Hilbert space
  $\mathcal{H}$ induced by $\kappa$ is dense in the space of
  continuous functions on $\mathcal{X}$ with respect to the supremum
  norm. Let $\mathcal{M}$ be a family of Borel probability measures on
  $\mathcal{X}$. A kernel $\kappa$ is characteristic for $\mathcal{M}$
  if the map $\mu \in \mathcal{M} \mapsto \int \kappa(\cdot, z) \mu(
  dz)$ is injective. If $\kappa$ is universal, then $\kappa$ is
  characteristic for any $\mathcal{M}$
  \citep{sriperumbudur11:_univer_charac_kernel_rkhs_embed_measur}. As
  an example, let $\mathcal{X}$ be a finite dimensional Euclidean
  space and define, for any $q \in (0,2)$, $k_{q}(x,y) =
  \tfrac{1}{2}(\|x\|^{q} + \|y\|^{q} - \|x - y\|^{q})$. The kernels
  $k_{q}$ are then characteristic for the collection of probability
  distributions with finite second moments
  \citep{lyons11:_distan,sejdinovic13:_equiv_rkhs}. In addition, by
  Eq.~\eqref{eq:5}, the maximum mean discrepancy with reproducing
  kernel $k_{q}$ can be written as
\begin{equation*}
  \mathrm{MMD}^{2}(F,Q;k_{q}) = 2 \mathbb{E} \|X - Y\|^{q} - 
  \mathbb{E}\|X - X'\|^{q} - \mathbb{E} \|Y - Y'\|^{q}.
\end{equation*}
 where $X, X'$ are
independent with distribution $F$, $Y, Y'$ are independent with
distribution $G$, and $X, Y$ are independent.  This coincides with the
notion of the energy distances of \citet{szekely13:_energ}, or, when $q
= 1$, a special case of the one-dimensional interpoint comparisons
of \citet{maa96:_reduc}.
\end{remark}

\begin{remark} The limiting distribution of
  $(m+n)U_{n,m}(\mathbf{X}, \mathbf{Y})$ under the null hypothesis of
  $F = G$ in Theorem~\ref{thm:mmd_unbiased_limiting} depends
  on the $\{\lambda_l\}$ which, in turn, depend on the distribution
  $F$; thus the limiting distribution is not distribution-free.
  Moreover the eigenvalues $\{\lambda_l\}$ can, at best, be estimated;
  for finite $n$, they cannot be explicitly determined when $F$ is
  unknown. In practice, generally the critical values are estimated
  through a bootstrap resampling or permutation test.
\end{remark}

\section{Main Results}
\label{sec:main-results}
We now address the nonparametric two-sample hypothesis tests of
\S~\ref{sec:hypothesis-tests} using the methodology described in 
\S~\ref{sec:kernel-based-two}. 
Throughout, we shall always assume that the distributions of the latent positions
satisfy the following distinct eigenvalues assumption. The assumption
implies that the estimates of the latent position obtained by the
adjacency spectral embedding in Definition~\ref{def:ase} will,
in the limit, be uniquely determined. 
\begin{assumption}\label{ass:rank_F}
  The distribution $F$ for the latent positions
  $X_1, X_2, \dots, \sim F$ is such that the second moment matrix
  $\Ex[X_1 X_1^{\top}]$ has $d$ distinct eigenvalues and $d$ is known.
\end{assumption}
The motivation behind this assumption is as follows: the matrix
$\Ex[X_1 X_1^{\top}]$ is of rank $d$ with $d$ known so
that given a graph $\mathbf{A} \sim \mathrm{RDPG}(F)$, one can
construct the adjacency spectral embedding of $\mathbf{A}$ into the
``right'' Euclidean space. The requirement that
$\Ex[X_1 X_1^{\top}]$ has $d$ distinct eigenvalues is due to the
intrinsic property of non-identifiability of random dot
product graphs, i.e., for any random dot product graph
$\mathbf{A}$, the latent position $\mathbf{X}$ associated with
$\mathbf{A}$ can only be estimated up to some true but unknown
orthogonal transformation. Because we are concerned with two-sample
hypothesis testing, we must guard against the scenario in which we
have two graphs $\mathbf{A}$ and $\mathbf{B}$ with latent positions
$\mathbf{X} = \{X_i\}_{i=1}^{n} \overset{\mathrm{i.i.d}}{\sim} F$ and
$\mathbf{Y} = \{Y_k\}_{k=1}^{m} \overset{\mathrm{i.i.d}}{\sim} F$ but
their estimates $\hat{\mathbf{X}}$ and $\hat{\mathbf{Y}}$ lie in
different, incommensurate subspaces of $\mathbb{R}^{d}$. That is to
say, the estimates $\hat{\mathbf{X}}$ and $\hat{\mathbf{Y}}$ satisfy
$\hat{\mathbf{X}} \approx \mathbf{X} \mathbf{W}_1$ and
$\hat{\mathbf{Y}} \approx \mathbf{Y} \mathbf{W}_2$, but
$\|\mathbf{W}_1 - \mathbf{W}_2\|_{F}$ does not converge to $0$ as $n,m
\rightarrow \infty$. See also \citet{fishkind15:_incom_phenom} for
exposition of a related so-called ``incommensurability phenomenon."


Indeed, we recognize that Assumption \ref{ass:rank_F} is restrictive; in
particular, it is not satisfied by the stochastic block model with
$K>2$ blocks of equal size and edge probabilities $p$ within
communities and $q$ between communities.  However, we are not aware of
any two-sample nonparametric inference procedure in which the
incommensurability problem is resolved, and Assumption
\ref{ass:rank_F} still permits two-sample nonparametric inference on a
wide class of random graphs. 
 
\begin{remark}
  This issue of incommensurability is an intrinsic feature of many
  dimension reduction techniques, and is not simply an artificial
  complication that arises in graph estimation. Consider, for example,
  principal component analysis in the following setting. Let
  $\mathbf{X}, \mathbf{Y} \in \mathbb{R}^{n\times d}$ and suppose that
  the rows of $\mathbf{X}$ and $\mathbf{Y}$ are i.i.d from some
  distribution $F$.  Furthermore, suppose that $\mathbf{X}$ and
  $\mathbf{Y}$ are unobserved, but instead $\mathbf{X}$ and
  $\mathbf{Y}$ are to be estimated or recovered from some higher
  dimension data
  $\mathbf{X}^{*} = [\mathbf{X} \mid \mathbf{Z}] \in \mathbb{R}^{n
    \times D},$
  and
  $\mathbf{Y}^{*} = [\mathbf{Y} \mid \mathbf{Z}'] \in \mathbb{R}^{n
    \times D}$,
  say via principal component analysis, where $\mathbf{Z}$ and
  $\mathbf{Z}'$ are $n \times (D - d)$ matrices whose rows are i.i.d
  from some other distribution $H$. That is to say, $\mathbf{X}$ is
  recovered via principal component analysis of $\mathbf{X}^{*}$ into
  $\mathbb{R}^{d}$ and similarly for $\mathbf{Y}$. Then depending on
  the covariance structure of $F$ and $H$, the recovered $\mathbf{X}$
  and $\mathbf{Y}$ could lie in incommensurate subspaces. 
 \end{remark}

\subsection{Two technical lemmas} \label{sec:2techlemma} We now state
two technical lemmas. The first lemma is the culmination 
of results from \citet{lyzinski13:_perfec} and \citet{tang14:_two}. The second lemma
lays the foundation for an empirical process result and is also a
central ingredient for showing the convergence to zero of a suitably
scaled version of our test statistic in the two-sample setting.
\begin{lemma}
  \label{lem:2}
  Let $(\mathbf{X},\mathbf{A}) \sim \mathrm{RDPG}(F)$
  be a $d$-dimensional random dot
  product graph on $n$ vertices with latent position distributions $F$
  satisfying the conditions in Assumption~\ref{ass:rank_F}. 
  Let $c > 0$ be
  arbitrary but fixed. There exists $n_0(c)$ such that if $n \geq n_0$
  and $\eta$ satisfies $n^{-c} < \eta < 1/4$, then there exists
  an orthogonal matrix $\mathbf{W}$
  dependent on $\mathbf{X}$ such that, with probability at least $1 -
  4 \eta$,
  \begin{gather}
    \label{eq:2}
    \|\hat{\mathbf{X}} - \mathbf{XW} \|_{F} \leq C_1,
 \\  
    \label{eq:7}
     \|\hat{\mathbf{X}} - \mathbf{XW}\|_{2 \rightarrow \infty} \leq C_{2}
    \sqrt{\frac{\log{(n/\eta)}}{n}} ,
  \end{gather}
  where $C_1$ and $C_2$ are constants depending only on $F$ and $n_0(c)$. 
\end{lemma} 

Lemma~\ref{lem:2} bounds 
the difference between
$\hat{\mathbf{X}}$ and $\mathbf{X}$ 
namely the Frobenius norm $\|\cdot\|_{F}$ and the maximum of the $l_2$
norms of the rows $\|\cdot\|_{2 \rightarrow \infty}$. The norm
$\|\cdot\|_{2 \rightarrow \infty}$ is induced by the vector norms
$\|\cdot\|_{2}$ and $\|\cdot\|_{\infty}$ via $\|\mathbf{A}\|_{2
  \rightarrow \infty} = \max_{\|\bm{x}\|_{2} = 1} \|\mathbf{A}
\bm{x}\|_{\infty}$. 
Eq.~\eqref{eq:7} follows from
Lemma~2.5 in \citet{lyzinski13:_perfec} while Eq.~\eqref{eq:2} 
follows from Theorem~2.3 in \citet{tang14:_two}. 

As a quick application of Lemma~\ref{lem:2},
suppose $(\mathbf{X},\mathbf{A}) \sim \mathrm{RDPG}(F)$ and 
$(\mathbf{Y}, \mathbf{B}) \sim \mathrm{RDPG}(G)$ where the
latent position distributions $F$ and $G$ satisfy the distinct eigenvalues
assumption and consider the hypothesis test of $\mathbb{H}_0 \colon F
\upVdash G$. Let $\kappa$ be a differentiable radial kernel and $U_{n,m}(\hat{\mathbf{X}},
\hat{\mathbf{Y}})$ is defined as
\begin{equation*}
  \begin{split}
    U_{n,m}(\hat{\mathbf{X}}, \hat{\mathbf{Y}}) &= \frac{1}{n(n-1)}
    \sum_{j \not = i}
    \kappa(\hat{X}_i,
    \hat{X}_j) - \frac{2}{mn} \sum_{i=1}^{n}
    \sum_{k=1}^{m} \kappa(\hat{X}_i,
    \hat{Y}_k) + \frac{1}{m(m-1)} \sum_{l \not = k} \kappa(\hat{Y}_k, \hat{Y}_l).
  \end{split}
\end{equation*}
Then there exists a deterministic unitary matrix $\mathbf{W}_0$ such that
\begin{equation*}
  U_{n,m}(\hat{\mathbf{X}}, \hat{\mathbf{Y}}) - U_{n,m}(\mathbf{X},
  \mathbf{Y} \mathbf{W}_0) \rightarrow 0
\end{equation*}
almost surely as $n, m \rightarrow \infty$. This can be seen as
follows. Let $\mathbf{W}_n$ and
$\mathbf{V}_m$ be orthogonal matrices in the eigendecomposition
$\mathbf{W}_n \mathbf{S}_1 \mathbf{W}_n = \mathbf{X}^{\top}
\mathbf{X}$, $\mathbf{V}_m \mathbf{S}_2 \mathbf{V}_m = \mathbf{Y}^{\top}
\mathbf{Y}$, respectively. Then 
 \begin{equation*}
  \begin{split}
  U_{n,m}(\hat{\mathbf{X}}, \hat{\mathbf{Y}}) - U_{n,m}(\mathbf{X} \mathbf{W}_n,
  \mathbf{Y} \mathbf{V}_m ) & = \frac{1}{n(n-1)}
    \sum_{j \not = i}
    (\kappa( \hat{X}_i,
     \hat{X}_j) - \kappa(\mathbf{W}_n X_i, \mathbf{W}_n X_j)) \\ &- \frac{2}{mn} \sum_{i=1}^{n}
    \sum_{k=1}^{m} (\kappa(\hat{X}_i,
    \hat{Y}_k) - \kappa(\mathbf{W}_n X_i, \mathbf{V}_m Y_k)) \\ &+ \frac{1}{m(m-1)} \sum_{l \not
      = k} \kappa( \hat{Y}_k,  \hat{Y}_l) -
    \kappa(\mathbf{V}_m Y_k, \mathbf{V}_m Y_l)).
    \end{split}
\end{equation*}
By differentiability of $\kappa$ and compactness of $\Omega$, we have
\begin{equation*}
  |\kappa( \hat{X}_i,  \hat{X}_j) - \kappa(\mathbf{W}_n X_i,
  \mathbf{W}_n X_j)| \leq C \max \{ \| \hat{X}_i - \mathbf{W}_n X_i \|, \|
  \hat{X}_j - \mathbf{W}_n X_j \| \} \leq C \|\hat{\mathbf{X}} -  
  \mathbf{X} \mathbf{W}_n \|_{2 \to \infty}. 
\end{equation*}
for some constant $C$ independent of $i$ and $j$. Similarly
\begin{gather*}
|\kappa(\hat{Y}_k, \hat{Y}_l) -
\kappa(\mathbf{V}_m Y_k,
  \mathbf{V}_m Y_l)| \leq C \|\hat{\mathbf{Y}} - 
  \mathbf{Y} \mathbf{V}_m \|_{2 \to \infty}, \\
  |\kappa( \hat{X}_i,  \hat{Y}_k) - \kappa(\mathbf{W}_n X_i,
  \mathbf{V}_m Y_k)| \leq C (\|\hat{\mathbf{X}} -
  \mathbf{X} \mathbf{W}_n \|_{2 \to \infty} + \|\hat{\mathbf{Y}}  -
  \mathbf{Y} \mathbf{V}_m \|_{2 \to \infty} ).
\end{gather*}
Thus \begin{equation*}
   |U_{n,m}(\hat{\mathbf{X}}, \hat{\mathbf{Y}}) - U_{n,m}(\mathbf{X}
   \mathbf{W}_n, 
  \mathbf{Y} \mathbf{V}_m) |  \leq 2C (\|\hat{\mathbf{X}} - 
  \mathbf{X} \mathbf{W}_n \|_{2 \to \infty} +  \|\hat{\mathbf{Y}}  -
  \mathbf{Y} \mathbf{V}_m \|_{2 \to \infty})
  \end{equation*}
  which converges, by Lemma~\ref{lem:2}, to zero almost surely as $n, m
  \rightarrow \infty$. Furthermore, $$U_{n,m}(\mathbf{X}
  \mathbf{W}_n, \mathbf{Y} \mathbf{V}_m) = U_{n,m}(\mathbf{X},
  \mathbf{Y} \mathbf{V}_m \mathbf{W}_n^{\top})$$ as $\kappa$ is a
  radial kernel. We have that 
\begin{align*}
n^{-1} \mathbf{X}^{\top}
  \mathbf{X} &= n^{-1} \mathbf{W}_1^{\top} \mathbf{S}_1
  \mathbf{W}_1 \textrm{ and }\\
m^{-1} \mathbf{Y}^{\top} \mathbf{Y} &=
  m^{-1} \mathbf{W}_2^{\top} \mathbf{S}_2 \mathbf{W}_2
\end{align*} 
are $\sqrt{n}$-consistent and $\sqrt{m}$-consistent estimators of $\mathbb{E}[X_1
  X_1^{\top}]$ and $\mathbb{E}[Y_1 Y_1^{\top}]$, respectively.  
  Since $F$ and $G$ satisfy the distinct
  eigenvalues condition, we can apply the Davis-Kahan theorem to each
  individual eigenvectors  of $\mathbb{E}[X_1
  X_1^{\top}]$ and $\mathbb{E}[Y_1 Y_1^{\top}]$, thereby showing that $\mathbf{W}_n$ and
  $\mathbf{V}_m$ are $\sqrt{n}$-consistent and $\sqrt{m}$-consistent
  estimator of the corresponding orthogonal matrices in the
  eigendecomposition of $\mathbb{E}[X_1 X_1^{\top}]$ and
  $\mathbb{E}[Y_1 Y_1^{\top}]$, respectively. If $F \upVdash
  G$, i.e., $F = G \circ \mathbf{W}_0$ for
  $\mathbf{W}_0$ orthogonal, then $\mathbf{V}_m
  \mathbf{W}_n^{\top} = \mathbf{W}_0 + O(\max\{n^{-1/2}, m^{-1/2}\})$
  and hence
  \begin{align*}
    |U_{n,m}(\hat{\mathbf{X}}, \hat{\mathbf{Y}}) - U_{n,m}(\mathbf{X}, 
  \mathbf{Y} \mathbf{W}_0)| = & |U_{n,m}(\hat{\mathbf{X}},
  \hat{\mathbf{Y}}) - U_{n,m}(\mathbf{X}, \mathbf{Y} \mathbf{V}_m
  \mathbf{W}_n^{\top})|\\
& + O(\max\{n^{-1/2}, m^{-1/2}\})
  \end{align*}
  which also converges to zero almost surely. That is to say, the test statistic based on the
  estimated latent position converges to the statistic based on the
  true but unknown latent positions. Thus one can construct, using the
  test statistics $U_{n,m}(\hat{\mathbf{X}}, \hat{\mathbf{Y}})$, a
  test procedure for $\mathbb{H}_0 \colon F \upVdash G$ that is {\em
    consistent} against all fixed alternatives $F \not \upVdash G$.
  This is in essence a first order result; in this regard, it is
  similar in spirit to first order consistency results
  for spectral clustering \citep{sussman12} and vertex classification
  \citep{sussman12:_univer}.  However, as we recall
  from Theorem~\ref{thm:mmd_unbiased_limiting}, in order to obtain a
  non-degenerate limiting distribution, we want to consider the scaled
  statistics $(m+n) U_{n,m}(\mathbf{X}, \mathbf{Y})$. Showing the
  convergence to zero of $(m+n)(U_{n,m}(\hat{\mathbf{X}},
  \hat{\mathbf{Y}}) - U_{n,m}(\mathbf{X}, \mathbf{Y} \mathbf{V}_m \mathbf{W}_n^{\top}))$
  is much more involved and is the main impetus behind the following
  lemma.
  
\begin{lemma}\label{lem:emp_proc}
  Let $\kappa$ be a twice continuously differentiable kernel. Let $\mathcal{F}_{\Phi} =
  \{ \Phi(Z) \colon Z \in \Omega \}$ where $\Phi(Z) =
  \kappa(\cdot, Z)$ is the feature map of $\kappa$, i.e., $f \in \mathcal{F}_{\Phi}$ if $f(X) = \kappa(X, Z)$ for some $Z$. 
  Suppose $(\mathbf{X}_n, \mathbf{A}_n) \sim \mathrm{RDPG}(F)$ for $n = 1,2,\dots$
  is a sequence of $d$-dimensional random dot product graphs and the
  latent positions distribution $F$ satisfies the distinct eigenvalues condition in
  Assumption~\ref{ass:rank_F}. Denote by $\mathbf{W}_n$ the orthogonal
  matrix in the eigendecomposition $\mathbf{W}_n \mathbf{S}_n
  \mathbf{W}_n^{\top} = \mathbf{X}_n^{\top} \mathbf{X}_n$. Then as $n
  \rightarrow \infty$, the sequence $\mathbf{W}_n$ satisfies
 $$\sup_{f \in \mathcal{F}_{\Phi}} \Bigl|\frac{1}{\sqrt{n}} \sum_{i=1}^n
\Bigl(f(\mathbf{W}_n \hat{X}_i) - f(X_i)\Bigr)\Bigr| \rightarrow 0$$
almost surely, where $\hat{\mathbf{X}}_n = \{\hat{X}_i\}_{i=1}^{n}$ is the adjacency spectral embedding of $\mathbf{A}_n$. 
\end{lemma}
Lemma~\ref{lem:emp_proc} is the main technical result of this
paper. 
Using the bound on $\|\hat{\mathbf{X}} - \mathbf{X}
\mathbf{W}\|_{2 \to \infty}$ from  Lemma~\ref{lem:2} implies that for some class of
continuous functions $\mathcal{F}$, e.g., continuous functions
 of the form $\phi(\|\cdot - c\|)$ for all $c$ in a compact subset
 of $\mathbb{R}^{d}$ , 
 there exists a sequence of orthogonal matrices $\mathbf{W}_n$
 such that
\begin{equation*}
  \sup_{f \in \mathcal{F}} \Bigl|\frac{1}{n} \sum_{i=1}^{n} (f(\mathbf{W}_{n} \hat{X}_i) -
  f(X_i)) \Bigr| \rightarrow 0
\end{equation*}
almost surely as $n \rightarrow \infty$ \citep[Theorem~15]{lyzinski13:_perfec}. Lemma~\ref{lem:emp_proc}
improves upon this; for some special class $\mathcal{F}$, 
the above also holds with the factor $1/n$ replaced by a factor of
$1/\sqrt{n}$. 

The proof of Lemma~\ref{lem:emp_proc} is given in the appendix.  A
rough sketch of the proof is as follows. For fixed
$f \in \mathcal{F}_{\Phi}$, a Taylor expansion allows one to write
$n^{-1/2} \sum_{i=1}^n \Bigl(f(\mathbf{W}_n \hat{X}_i) - f(X_i)\Bigr)$ in terms
of
$\sum_{i} \lambda_{i}^{-1/2} \bm{v}_{i}^{\top}(\mathbf{A} -
\mathbf{P}) \bm{u}_i$
for unit vectors $\bm{v}_i$ depending on $f$ and $\bm{u}_i$ depending
on $\{X_i\}$; here $\lambda_{i}$ are the eigenvalues of
$\mathbf{P}$. Hoeffding's inequality applied to the sum
$\sum_{i} \lambda_{i}^{-1/2} \bm{u}_{i}^{\top}(\mathbf{A} -
\mathbf{P}) \bm{v}_i$
provides an exponential tail bound for each
$f \in \mathcal{F}_{\Phi}$. A chaining argument similar to that in
\citet[Section~3.2]{geer00:_empir_m} and bounds for the so-called {\em
  covering number} of $\mathcal{F}_{\Phi}$ (again, see
\citet[\S~2.3]{geer00:_empir_m} for a precise definition) lead to an
exponential tail bound that is uniform over all
$f \in \mathcal{F}_{\Phi}$.

The application of Lemma~\ref{lem:emp_proc} to our nonparametric
two-sample hypothesis testing problem is presented in
\S~\ref{sec:equality-case}. Another interesting consequence of Lemma~\ref{lem:emp_proc} is a functional central limit
theorem for $\hat{\mathbf{X}}$, which is the topic of the following
subsection. 

\subsection{A functional central limit theorem for $\hat{\mathbf{X}}$} 
\label{sec:funct-centr-limit}
By replacing the class of functions $\mathcal{F}_{\Phi}$ in
Lemma~\ref{lem:emp_proc} with a more general class of functions
$\mathcal{F}$ whose covering numbers are still ``small," a similar chaining
argument can be adapted to yield the following functional central
limit theorem. (For a comprehensive discussion of functional central
limit theorems, see, for example, \citet{dudley_unif,vaart96:_weak} and
the references therein.) We first recall certain definitions, which we
reproduce from \citet{vaart96:_weak}. Let $X_i, 1 \leq i \leq n$ be
identically distributed random variables on a measure space
$(\mathcal{X}, \mathcal{B})$, and let $\mathbb{P}_n$ be their
associated {\em empirical measure}; that is, $\mathbb{P}_n$ is the
discrete random measure defined, for any $E \in \mathcal{B}$, by
$$\mathbb{P}_n(E)=\frac{1}{n} \sum_{i=1}^n 1_{E}(X_i).$$
Let $P$ denote the common distribution of the random variables $X_i$,
and suppose that $\mathcal{F}$ is a class of measurable, real-valued
functions on $\mathcal{X}$.  The {\em $\mathcal{F}$-indexed empirical
  process} $\mathbb{G}_n$ is the stochastic process
\begin{equation*}
f \mapsto \mathbb{G}_n(f)=\sqrt{n}(\mathbb{P}_n -P)f=
\frac{1}{\sqrt{n}} \sum_{i=1}^n \Bigl(f(X_i)-
\mathbb{E}[f(X_i)]\Bigr). 
\end{equation*}
Under certain conditions, the empirical process $\{\mathbb{G}_n(f): f
\in \mathcal{F}\}$ can be viewed as a map into
$\ell^{\infty}(\mathcal{F})$, the collection of all uniformly bounded
real-valued functionals on $\mathcal{F}$.  In particular, let
$\mathcal{F}$ be a class of functions for which the empirical process
$\mathbb{G}_n=\sqrt{n}(\mathbb{P}_n-P)$ converges to a limiting
process $\mathbb{G}$ where $\mathbb{G}$ is a tight Borel-measurable
element of $\ell^{\infty}(\mathcal{F})$ (more specifically a Brownian
bridge). Then $\mathcal{F}$ is said to be a {\em $P$-Donsker class},
or for brevity, $P$-Donsker \citep[\S~2.1]{vaart96:_weak}. A sufficient condition, albeit a rather
strong one, for $\mathcal{F}$ to be $P$-Donsker is via the entropy for
the supremum norm. That is, let $N_{\infty}(\delta, \mathcal{F})$ be
the smallest value of $N$ such that there exists $\{f_j\}_{j=1}^{N}$
with $\sup_{f \in \mathcal{F}} \min_{j} \|f - f_j \|_{\infty} \leq
\delta$. Then $\mathcal{F}$ is $P$-Donsker for any $P$ if \citep[\S~2.5.2]{vaart96:_weak}
\begin{equation}
  \label{eq:9}
  \int_{0}^{\infty} \sqrt{\log{N_{\infty}(\delta, \mathcal{F})}} \,\,
    \mathrm{d} \delta < \infty.
\end{equation}
As an example, let $\mathcal{F}$ be the unit ball associated with a
kernel $\kappa$ on a compact $\Omega \subset \mathbb{R}^{d}$. Then
$\mathcal{F}$ is $P$-Donsker provided $\kappa$ is $m$-times continuously
differentiable on $\Omega$ for some $m \geq 2d + 1$ \citep[Theorem~2.7.1 \&
Theorem 2.5.6]{vaart96:_weak}. 
The unit ball associated with the Gaussian kernel on $\mathbb{R}^{d}$
is thus $P$-Donsker for all $d$.

\begin{theorem}
  \label{thm:u-statistics}
  Let $({\bf X}_n, {\bf A}_n)$ for $n = 1,2,\dots,$ be a sequence of $d$-dimensional
  $\mathrm{RDPG}(P)$ where the latent position distribution $P$
  satisfies the distinct eigenvalues condition in
  Assumption~\ref{ass:rank_F}. Let $\mathcal{F}$ be a collection of
  (at least) twice continuously differentiable functions on $\Omega$ with
  \begin{equation*}
    \sup_{f \in \mathcal{F}, X \in \Omega} \|(\partial f)(X)\| <
    \infty; \qquad \sup_{f \in \mathcal{F}, X \in \Omega}
    \| (\partial^{2} f)(X) \| < \infty.
  \end{equation*} 
  Furthermore, suppose 
  $\mathcal{F}$ satisfies Eq.~\eqref{eq:9} so that
  $\mathbb{G}_n=\sqrt{n}(\mathbb{P}_n-P)$ converges to $\mathbb{G}$, a
  $P$-Brownian bridge on $\ell^{\infty}(\mathcal{F})$. Denote by
  $\mathbf{W}_n$ the orthogonal matrices in the eigendecomposition
  $\mathbf{W}_n \mathbf{S}_n \mathbf{W}_n^{\top} = \mathbf{X}_n^{\top}
  \mathbf{X}_n$. Then as $n \rightarrow \infty$, the
  $\mathcal{F}$-indexed empirical process
  \begin{equation}
    f \in \mathcal{F} \mapsto \hat{\mathbb{G}}_{n} f =
    \frac{1}{\sqrt{n}}\sum_{i=1}^{n} \Bigl(f( \mathbf{W}_n \hat{X}_i) - \mathbb{E}[f(X_i)]\Bigr)
  \end{equation}
  also converges to $\mathbb{G}$ on $\ell^{\infty}(\mathcal{F})$.
\end{theorem}
Theorem~\ref{thm:u-statistics} is in essence a functional central
limit theorem for the estimated latent positions $\{\hat{X}_i\}$ in
the random dot product graph setting. We emphasize that for any $n$, the
$\{\hat{X}_i\}_{i=1}^{n}$ are not jointly independent random variables, i.e., 
Theorem~\ref{thm:u-statistics} is a functional central limit
theorem for {\em dependent} data. Due to the non-identifiability of random dot
product graphs, there is an explicit dependency on the sequence of
orthogonal matrices $\mathbf{W}_n$; note, however, that $\mathbf{W}_n$
depends solely on $\mathbf{X}_n$ and not on the $\{\hat{X}_i\}$.
\subsection{Consistent Testing}
\label{sec:equality-case}
We now consider testing the hypothesis $\mathbb{H}_0 \colon F
\upVdash G$ using the kernel-based framework of
\S~\ref{sec:kernel-based-two}. For our purpose, we shall assume
henceforth that $\kappa$ is a twice continuously-differentiable radial
kernel and that $\kappa$ is also universal. Examples of such kernels
are the Gaussian kernels and the inverse multiquadric kernels
$\kappa(x,y) = (c^{2} + \|x - y\|^2)^{-\beta}$ for $c,\beta > 0$. 

To justify this assumption on our kernel, we remark that in Theorem~\ref{thm:mmd_unbiased_ase} below, we show that the test
statistic $U_{n,m}(\hat{\mathbf{X}}, \hat{\mathbf{Y}})$ based on the
estimated latent positions converges to the corresponding statistic
$U_{n,m}(\mathbf{X}, \mathbf{Y})$ for the true but unknown latent
positions. Due to the non-identifiability of the random dot product
graph under unitary transformation, \emph{any} estimate of the latent
positions is close, only up to an appropriate orthogonal transformations, to
$\mathbf{X}$ and $\mathbf{Y}$. We have seen in \S~\ref{sec:2techlemma} that
for a radial kernel, this implies the approximations
$\kappa(\hat{X}_i, \hat{X}_j) \approx \kappa(X_i, X_j)$,
$\kappa(\hat{Y}_k,\hat{Y}_l) \approx \kappa(Y_k, Y_l)$ and the
convergence of $U_{n,m}(\hat{\mathbf{X}}, \hat{\mathbf{Y}})$ to
$U_{n,m}(\mathbf{X}, \mathbf{Y})$. If $\kappa$ is
not a radial kernel, the above approximations might not hold and
$U_{n,m}(\hat{\mathbf{X}}, \hat{\mathbf{Y}})$ need not converge to
$U_{n,m}(\mathbf{X}, \mathbf{Y})$. The assumption that
$\kappa$ is twice continuously-differentiable is for the technical
conditions of Lemma~\ref{lem:emp_proc}.  Finally, the
assumption that $\kappa$ is universal allows the test procedure to be
consistent against a large class of alternatives.

 \begin{theorem}
  \label{thm:mmd_unbiased_ase}
  Let $(\mathbf{X}, \mathbf{A}) \sim \mathrm{RDPG}(F)$ and
  $(\mathbf{Y}, \mathbf{B}) \sim \mathrm{RDPG}(G)$ be independent
  random dot product graphs with latent position distributions $F$ and
  $G$. Furthermore, suppose that both $F$ and $G$ satisfies the
  distinct eigenvalues condition in Assumption~\ref{ass:rank_F}.
  Consider the hypothesis test
  \begin{align*}
    H_{0} \colon F \upVdash G \quad 
    \text{against} \quad H_{A} \colon F  \nupVdash G.
  \end{align*}
  Denote by $\hat{\mathbf{X}} = \{\hat{X}_1, \dots, \hat{X}_n\}$
  and $\hat{\mathbf{Y}} = \{\hat{Y}_1, \dots, \hat{Y}_m\}$ the
  adjacency spectral embedding of $\mathbf{A}$ and
  $\mathbf{B}$, respectively.
  Let $\mathbf{W}_1$
  and $\mathbf{W}_{2}$
  be $d \times d$ orthogonal matrices in the eigendecomposition
  $\mathbf{W}_1 \mathbf{S}_1 \mathbf{W}_{1}^{\top} = \mathbf{X}^{\top}
  \mathbf{X}$, $\mathbf{W}_{2} \mathbf{S}_{2} \mathbf{W}_{2} =
  \mathbf{Y}^{\top} \mathbf{Y}$, respectively.  
Suppose that $m, n \rightarrow \infty$
  and $m/(m+n) \rightarrow \rho \in (0,1)$. Then under the null
  hypothesis of $F \upVdash G$, the sequence of matrices $\mathbf{W}_{n,m} = \mathbf{W}_{2} \mathbf{W}_{1}^{\top}$ satisfies
  \begin{equation}
    \label{eq:conv_mmdXhat_null}
    (m+n) (U_{n,m}(\hat{\mathbf{X}}, \hat{\mathbf{Y}}) -
    U_{n,m}(\mathbf{X}, \mathbf{Y} \mathbf{W}_{n,m})) \overset{\mathrm{a.s.}}{\longrightarrow} 0.
  \end{equation}
  Under the alternative hypothesis of $F \nupVdash G$, the sequence of
  matrices ${\bf W}_{n,m} $ satisfies
 \begin{equation}
   \label{eq:conv_mmdXhat_alt}
   \frac{m+n}{\log^2{\!(m+n)}} 
   (U_{n,m}(\hat{\mathbf{X}}, \hat{\mathbf{Y}}) - U_{n,m}(\mathbf{X},
   \mathbf{Y} \mathbf{W}_{n,m})) \overset{\mathrm{a.s.}}{\longrightarrow} 0.
  \end{equation}
\end{theorem}
\begin{proof}
We first define the statistic $V_{n,m}(\mathbf{X}, \mathbf{Y})$
   \begin{equation}
     \label{eq:13}
    \begin{split}
V_{n,m}(\mathbf{X}, \mathbf{Y}) &= \Bigl\|
    \frac{1}{n} \sum_{i=1}^{n} \Phi(X_i) - \frac{1}{m} \sum_{k=1}^{m}
    \Phi(Y_k) \Bigr \|_{\mathcal{H}}^{2} \\ 
    &= 
    \frac{1}{n^2} \sum_{i=1}^{n} \sum_{j=1}^{n} \kappa(X_i,X_j)
    - \frac{2}{mn} \sum_{i=1}^{n} \sum_{k=1}^{m} \kappa(X_i, Y_k) +
    \frac{1}{m^2} \sum_{k=1}^{m} \sum_{l=1}^{m} \kappa(Y_k, Y_l).
    \end{split}
  \end{equation}
  We shall prove that the difference
  \begin{equation}
    \label{eq:16}
     (m+n)(V_{n,m}(\hat{\mathbf{X}}, \hat{\mathbf{Y}}) -
  V_{n,m}(\mathbf{X}, \mathbf{Y} \mathbf{W}_{n,m})) \overset{\mathrm{a.s}}{\longrightarrow} 0
  \end{equation}
  under the hypothesis $F
  \upVdash G$. The claim $(m+n) (U_{n,m}(\hat{\mathbf{X}},
  \hat{\mathbf{Y}}) - U_{n,m}(\mathbf{X}, \mathbf{Y}
  \mathbf{W}_{n,m})) \overset{\mathrm{a.s.}}{\longrightarrow}
  0$ in Theorem~\ref{thm:mmd_unbiased_ase} follows from
  Eq.~\eqref{eq:16} and the following expression
  \begin{equation*}
   (m+n)(V_{n,m}(\hat{\mathbf{X}}, \hat{\mathbf{Y}}) - V_{n,m}(\mathbf{X},
  \mathbf{Y} \mathbf{W}_{n,m})) = (m+n) (U_{n,m}(\hat{\mathbf{X}}, \hat{\mathbf{Y}}) -
  U_{n,m}(\mathbf{X}, \mathbf{Y} \mathbf{W}_{n,m})) + r_1 + r_2
  \end{equation*}
  where $r_1$ and $r_2$ are defined as (recall that $\kappa$ is a
  radial kernel)
   \begin{gather*}
     r_1 = \frac{m+n}{n(n-1)} \sum_{i=1}^{n} \Bigl(\kappa(X_i, X_i)
     - \kappa(\hat{X}_i, \hat{X}_i)\Bigr) 
     + \frac{m+n}{m(m-1)} \sum_{k=1}^{m} \Bigl(\kappa(Y_k, Y_k) -
     \kappa(\hat{Y}_k, \hat{Y}_k)\Bigr), \\
    r_2 = \frac{m+n}{n^2(n-1)} \sum_{i=1}^{n} \sum_{j=1}^{n}
     \Bigl(\kappa(X_i, X_j) - \kappa(\hat{X}_i, \hat{X}_j)\Bigr) + \frac{m+n}{m^2(m-1)}\sum_{k=1}^{m}
     \sum_{l=1}^{m} \Bigl(\kappa(Y_k, Y_l) - \kappa(\hat{Y}_k, \hat{Y}_l)\Bigr).
   \end{gather*}
   As $\kappa$ is twice continuously differentiable, we can show,
   by the compactness of $\Omega$ and the bounds in Lemma~\ref{lem:2} 
   that both $r_1$ and $r_2$ converges to
   $0$ almost surely. In particular, there exists a constant $L$ 
   such that both $|r_1|$ and $|r_2|$ is bounded from
   above by
    \begin{equation*}
        L(m+n) \biggl\{ \frac{\|
         \hat{\mathbf{X}} - \mathbf{X} \mathbf{W}_1 \|_{2 \rightarrow \infty}}{n-1} +
       \frac{\|\hat{\mathbf{Y}}  - \mathbf{Y} \mathbf{W}_2  \|_{2 \rightarrow
           \infty}}{m-1}\biggr\}.
   \end{equation*}

  We thus proceed to establishing Eq.~\eqref{eq:16}. Define $\xi_{W},
   \hat{\xi} \in \mathcal{H}$ by
  \begin{gather*}
    \xi_{W} = \frac{\sqrt{m+n}}{n} \sum_{i=1}^{n} \kappa(\mathbf{W}_1 X_i, \cdot) -
    \frac{\sqrt{m+n}}{m} \sum_{k=1}^{m} \kappa(\mathbf{W}_{2} Y_k, \cdot); \\
    \hat{\xi} = \frac{\sqrt{m+n}}{n} \sum_{i=1}^{n}
    \kappa( \hat{X}_i, \cdot) -
    \frac{\sqrt{m+n}}{m} \sum_{k=1}^{m} \kappa( \hat{Y}_k, \cdot).
  \end{gather*}
  Note that
  \begin{align*}
    \Bigl|(m+n)(V_{n,m}(\hat{\mathbf{X}}, \hat{\mathbf{Y}}) -
    V_{n,m}(\mathbf{X}, \mathbf{Y} \mathbf{W}_{n,m}))\Bigr| 
&= \Bigl|\|\xi_{W}\|_{\mathcal{H}}^{2} -
    \|\hat{\xi} \|_{\mathcal{H}}^{2}\Bigr|\\
& \leq \|\xi_{W} -
    \hat{\xi} \|_{\mathcal{H}}\Bigl(2 \|\xi_{W} \|_{\mathcal{H}} +
    |\xi_{W} - \hat{\xi} \|_{\mathcal{H}}\Bigr).
  \end{align*}
  We now bound the terms $\|\xi_{W} - \hat{\xi} \|_{\mathcal{H}}$ and
  $\|\xi_{W}\|_{\mathcal{H}}$. We first bound 
  $\|\xi_{W}\|_{\mathcal{H}}$. Let $\mathbf{T}_1$ and $\mathbf{T}_2$
  be the orthogonal matrices in the eigendecomposition of
  $\mathbb{E}[X_1 X_1^{\top}]$ and $\mathbb{E}[Y_1 Y_1^{\top}]$. The distinct eigenvalues condition in
  Assumption~\ref{ass:rank_F} implies, by the Davis-Kahan theorem, that $\mathbf{W}_1 = \mathbf{T}_1 + O(n^{-1/2})$ and $\mathbf{W}_2 =
  \mathbf{T}_2 + O(m^{-1/2})$. When $F
  \upVdash G$, $F \circ \mathbf{T}_1 = G \circ \mathbf{T}_2$ and hence
  by adding and subtracting terms, we have
  \begin{equation*}
    \xi_{W} = \sqrt{\frac{m+n}{n}} \sum_{i=1}^{n}
    \frac{\kappa(\mathbf{T}_1 X_i,
      \cdot) -  \mu[F \circ \mathbf{T}_1]}{\sqrt{n}} - \sqrt{\frac{m+n}{m}}
    \sum_{k=1}^{m} \frac{\kappa(\mathbf{T}_2 Y_k, \cdot) -
      \mu[G \circ \mathbf{T}_2]}{\sqrt{m}} + O(1).
  \end{equation*}
  That is, $\xi_W -
  O(1)$ is a sum of independent mean zero random elements of
  $\mathcal{H}$. In addition $\|\kappa(Z,
  \cdot) - \mu[F]\|_{\mathcal{H}} \leq 2$ for any $Z \in
  \mathbb{R}^{d}$.  Using a Hilbert space concentration inequality
  \citep[Theorem~3.5]{pinelis94:_optim_banac}, we obtain that
  \begin{equation*}
    \mathbb{P}[\|\xi_{W} \|_{\mathcal{H}} \geq \sqrt{m+n}(s/\sqrt{n}
    + t/\sqrt{m})] \leq 2 \Bigl(\exp(-(1+m/n)s^2/8) + \exp(-(1+n/m)t^2/8)\Bigr),
  \end{equation*}
  which implies that $\|\xi_{W}\|_{\mathcal{H}}$ is bounded in probability. 
 We now bound $\|\xi_{W} - \hat{\xi} \|_{\mathcal{H}}$. We have
  \begin{equation*}
    \xi_{W} - \hat{\xi} = \sqrt{\frac{m+n}{n}} \sum_{i=1}^{n}
    \frac{\kappa(\mathbf{W}_1 X_i, \cdot) - \kappa( \hat{X}_i, \cdot)}{\sqrt{n}}
    - \sqrt{\frac{m+n}{n}} \sum_{k=1}^{m} \frac{\kappa(\mathbf{W}_2 Y_k,\cdot)
      - \kappa( \hat{Y}_k, \cdot)}{\sqrt{m}}
  \end{equation*}
and Lemma~\ref{lem:emp_proc} implies (as $\kappa$ is radial)
\begin{equation*}
  \sqrt{\frac{m+n}{n}} \sum_{i=1}^{n}
    \frac{\kappa( \mathbf{W}_1 X_i, \cdot) - \kappa(\hat{X}_i,
      \cdot)}{\sqrt{n}} \overset{\mathrm{a.s.}}{\longrightarrow} 0; \qquad 
    \sqrt{\frac{m+n}{n}} \sum_{k=1}^{m} \frac{\kappa(\mathbf{W}_2 Y_k,\cdot)
      - \kappa(\hat{Y}_k, \cdot)}{\sqrt{m}}
    \overset{\mathrm{a.s.}}{\longrightarrow} 0
\end{equation*}
as $m, n \rightarrow \infty$, $m/n \rightarrow \rho \in
(0,1)$. Thus $\|\xi_{W} - \hat{\xi}\|_{\mathcal{H}} \rightarrow 0$
and Eq.~\eqref{eq:16} and Eq.~\eqref{eq:conv_mmdXhat_null} are
established. 

We now derive Eq.~\eqref{eq:conv_mmdXhat_alt}. We note that in the
case when $F \nupVdash G$, one still has
\begin{equation*}
    \Bigl|(m+n)(V_{n,m}(\hat{\mathbf{X}}, \hat{\mathbf{Y}}) -
    V_{n,m}(\mathbf{X}, \mathbf{Y} \mathbf{W}_{n,m}))\Bigr| \leq \|\xi_{W} -
    \hat{\xi} \|_{\mathcal{H}}\Bigl(2 \|\xi_{W} \|_{\mathcal{H}} +
    \|\xi_{W} - \hat{\xi} \|_{\mathcal{H}}\Bigr)
\end{equation*}
where $\hat{\xi}$ and $\xi_W$ are defined identically to the case when
$F \upVdash G$. However, when $F \nupVdash G$, the bound $\|\xi_{W}\|_{\mathcal{H}} =
O(1)$ with high probability no longer holds. Indeed, when $F \nupVdash G$, 
\begin{equation*}
  \xi_W - O(1) =  \frac{\sqrt{m+n}}{n} \sum_{i=1}^{n} \kappa(\mathbf{T}_1 X_i, \cdot) -
    \frac{\sqrt{m+n}}{m} \sum_{k=1}^{m} \kappa(\mathbf{T}_{2} Y_k,
    \cdot)
\end{equation*}
is not a sum of mean $0$ random variables. We thus bound
$\|\xi_W\|_{\mathcal{H}} = O(\sqrt{n \log{n}})$ with high probability. 
The proof of Lemma~3 yields $\|\hat{\xi} - \xi_W \|_{\mathcal{H}} = O(n^{-1/2}
\log{n})$ with high probability (see Eq.~\eqref{eq:supBnd} in
the appendix). Hence $\Bigl|(m+n)(V_{n,m}(\hat{\mathbf{X}}, \hat{\mathbf{Y}}) -
    V_{n,m}(\mathbf{X}, \mathbf{Y} \mathbf{W}_{n,m}))\Bigr|$ is of
    order $\log^{3/2}{n}$ with high probability and
    Eq.~\eqref{eq:conv_mmdXhat_alt} follows.  
\end{proof}

Eq.\eqref{eq:conv_mmdXhat_null} and Eq.\eqref{eq:conv_mmdXhat_alt}
state that the test statistic $U_{n,m}(\hat{\mathbf{X}},
\hat{\mathbf{Y}})$ using the {\em estimated} latent positions is
almost identical to the statistic $U_{n,m}(\mathbf{X}, \mathbf{Y} \mathbf{W}_{n,m})$
defined in Eq.~\eqref{eq:10} using the true latent positions, under
both the null and alternative hypothesis. Because $\kappa$ is a universal
kernel, $U_{n,m}(\mathbf{X}, \mathbf{Y} \mathbf{W}_{n,m})$ converges
to $0$ under the null and converges to a positive number under the
alternative. The test statistic $U_{n,m}(\hat{\mathbf{X}},
\hat{\mathbf{Y}})$ therefore yields a test procedure that is
consistent against any alternative, provided that both $F$ and $G$
satisfy Assumption~\ref{ass:rank_F}, namely that the second moment
matrices have $d$ distinct eigenvalues.

\begin{figure}[htb!]
  \centering
  \subfloat[][$(\mathbf{X}, \mathbf{A})
\sim \mathrm{RDPG}(F)$, $(\mathbf{Y}, \mathbf{B})
\sim \mathrm{RDPG}(F)$]{\includegraphics[width=0.5\textwidth]{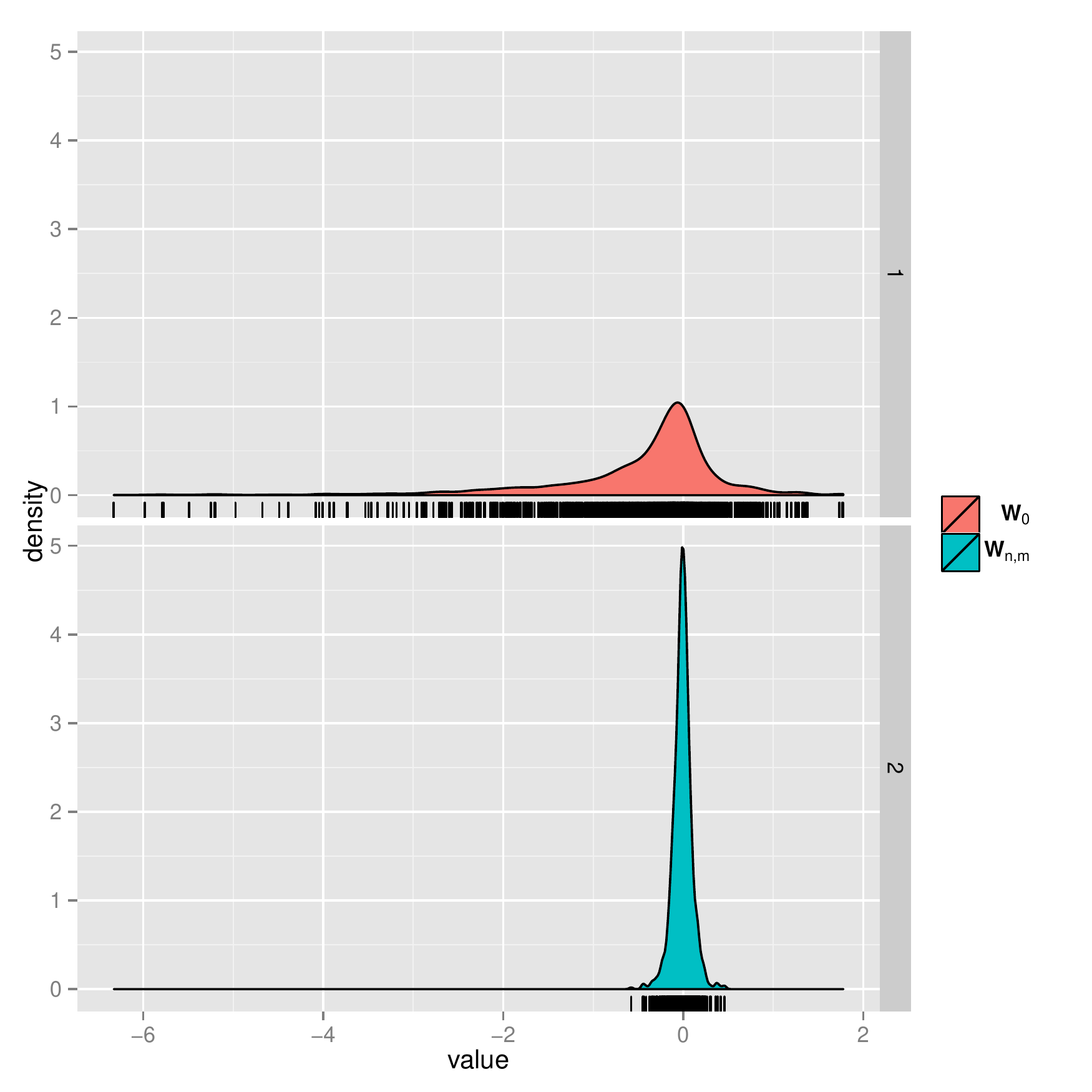}}
\hfill 
  \subfloat[][$(\mathbf{X}, \mathbf{A})
\sim \mathrm{RDPG}(F)$, $(\mathbf{Y}, \mathbf{B})
\sim \mathrm{RDPG}(G)$]{\includegraphics[width=0.5\textwidth]{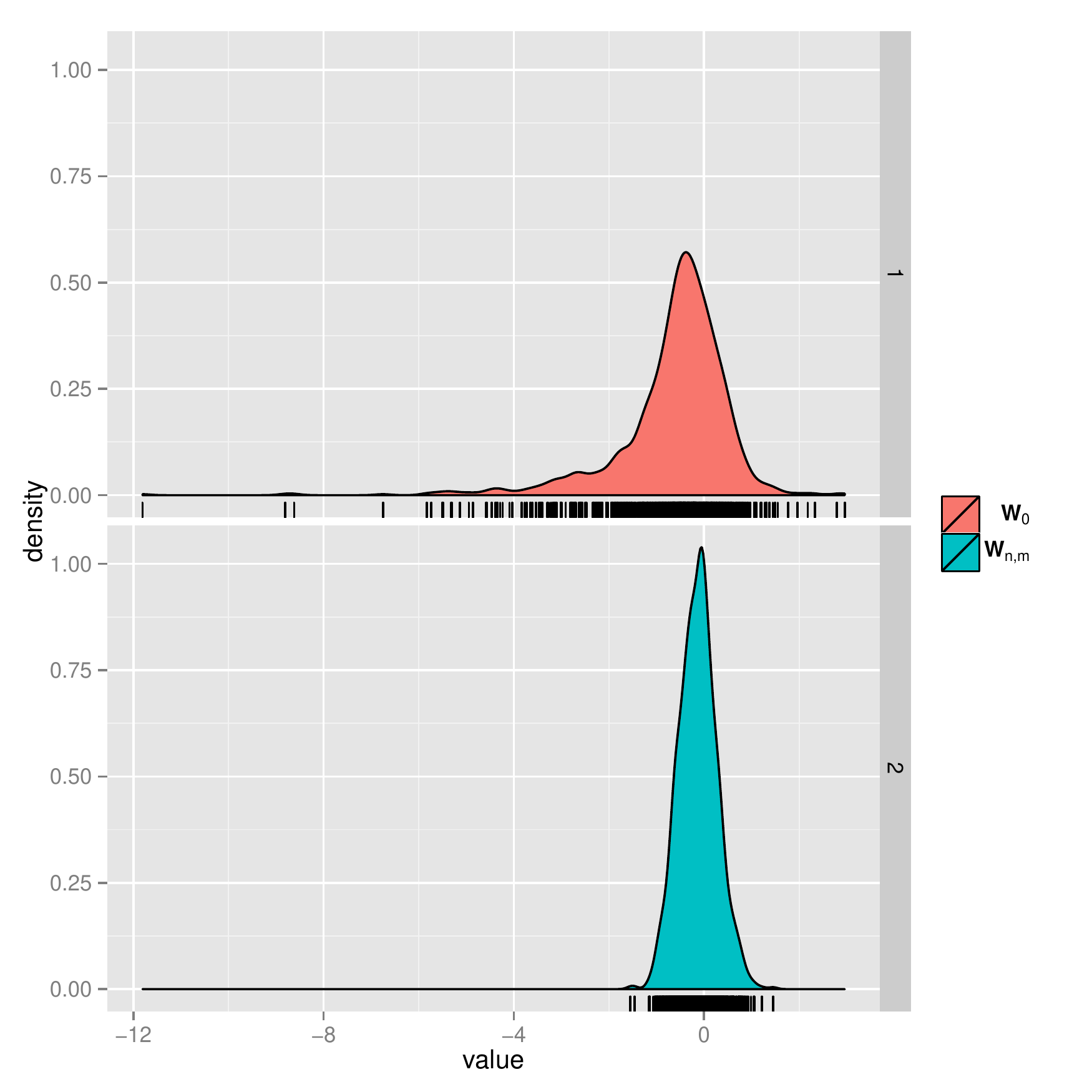}}
  \caption{Comparison between the random $\mathbf{W}_{n,m}$ and fixed
    but unknown $\mathbf{W}_0$. The empirical distributions of
$(m+n)(U_{n,m}(\hat{\mathbf{X}}, \hat{\mathbf{Y}}) -
U_{n,m}(\mathbf{X}, \mathbf{Y} \mathbf{W}_{0}))$ (in red) and
$(m+n)(U_{n,m}(\hat{\mathbf{X}}, \hat{\mathbf{Y}}) -
U_{n,m}(\mathbf{X}, \mathbf{Y} \mathbf{W}_{n,m})$ (in blue) under (a) the null setting of
$(\mathbf{X}, \mathbf{A}) \sim F, (\mathbf{Y}, \mathbf{B}) \sim F$ and (b) the
alternative setting of $(\mathbf{X}, \mathbf{A}) \sim F, (\mathbf{Y},
\mathbf{B}) \sim G$.}
  \label{fig:nullW}
\end{figure}

We note that a subtle point in
the statement and argument of the theorem is that $\mathbf{W}_{n,m}$
is a random quantity depending on $\mathbf{X}_n$ and
$\mathbf{Y}_m$. There does exist a deterministic matrix $\mathbf{W}_0$
depending only on $F$ and $G$ such that $\mathbf{W}_{n,m} \rightarrow
\mathbf{W}_0$ almost surely as $m,n \rightarrow \infty$. Indeed, from
the proof of the theorem, we have that $\mathbf{W}_1$ is a $\sqrt{n}$-consistent
  estimator of $\mathbf{T}_1$ where $\mathbf{T}_1$ is the orthogonal
  matrix in the eigendecomposition of $\mathbb{E}[X_1
  X_1^{\top}]$ and that $\mathbf{W}_2$ is a $\sqrt{m}$-consistent
  estimator of $\mathbf{T}_2$ where $\mathbf{T}_2$ is the orthogonal
  matrix in the eigendecomposition of $\mathbb{E}[Y_1
  Y_1^{\top}]$. Under the null hypothesis, $F \circ
  \mathbf{T}_1 = G \circ \mathbf{T}_2$; hence if we define $\mathbf{W}_0$ as
  $\mathbf{T}_2 \mathbf{T}_1^{\top}$, then $\mathbf{W}_2
  \mathbf{W}_1^{\top}$ is a $\sqrt{n}$-consistent estimator of
  $\mathbf{W}_0$. This
convergence of order $O(n^{-1/2})$ is, however, not
sufficiently fast to guarantee that $(m+n)(U_{n,m}(\hat{\mathbf{X}},
\hat{\mathbf{Y}}) - U_{n,m}(\mathbf{X}, \mathbf{Y} \mathbf{W}_{0}))$
converges to zero almost surely when $F \upVdash G$. 
For example, let $F$ be a mixture of
two multivariate logit-normal distributions with mean parameters
$(0,0), (4,4)$, identity covariance matrices and mixture components
$(0.4,0.6)$; let $G$ be a multivariate logit-normal distribution with
mean parameter $(2,2)$ and identity covariance matrix.
Figure~\ref{fig:nullW} 
illustrates that the difference 
$(m+n)(U_{n,m}(\hat{\mathbf{X}}, \hat{\mathbf{Y}}) -
U_{n,m}(\mathbf{X}, \mathbf{Y} \mathbf{W}_{n,m}))$ is in general
smaller compared to the difference
$(m+n)(U_{n,m}(\hat{\mathbf{X}}, \hat{\mathbf{Y}}) -
U_{n,m}(\mathbf{X}, \mathbf{Y} \mathbf{W}_0))$, thereby complicating the
derivation of the exact nondegenerate limiting distribution for $(m+n)
U_{m,n}(\hat{{\bf X}}, \hat{{\bf Y}})$. Nevertheless, since the nondegenerate limiting distribution
for $(m+n) U_{m,n}(\hat{{\bf X}}, \hat{{\bf Y}})$ will not be
distribution-free, the fact that it is currently unknown 
is, for all practical purposes, irrelevant. Indeed, the proposed test statistic still yields a
consistent test procedure whose critical values can be obtained
through a simple bootstrapping procedure.

\begin{remark}
  The computational cost for implementing the test procedure in
  Theorem~\ref{thm:mmd_unbiased_ase} consist
  mainly of two parts, namely computing the adjacency spectral
  embedding of the graphs $\mathbf{A}$ and $\mathbf{B}$, and computing
  the test statistic $U_{n,m}(\hat{\mathbf{X}},
  \hat{\mathbf{Y}})$.
  Assuming $n \geq m$, the adjacency spectral embedding of
  $\mathbf{A}$ and $\mathbf{B}$ into $\mathbb{R}^{d}$ is a (partial)
  singular value decomposition of $\mathbf{A}$ and $\mathbf{B}$ and thus can be computed in
  $O(n^{2} d)$ time. The test statistic
  $U_{n,m}(\hat{\mathbf{X}}, \hat{\mathbf{Y}})$ can be evaluated in
  $O(n^{2})$ time.
\end{remark} 

\begin{remark}
  The proof of Theorem~\ref{thm:mmd_unbiased_ase} can be adapted
  to show that data-adaptive bandwidth selections behave similarly for
  $\hat{\mathbf{X}}$ and $\hat{\mathbf{Y}}$ as for $\mathbf{X}$ and
  $\mathbf{Y}$. That is to say, we can show that under the null
  hypothesis, $\Delta_{\theta} = (m+n) (U_{n,m}(\hat{\mathbf{X}},
  \hat{\mathbf{Y}}) - U_{n,m}(\mathbf{X},
  \mathbf{Y}\mathbf{W}_{n,m}))$ converges to $0$ uniformly over some
  family of kernels $\{\kappa_{\theta} \colon \theta \in
  \Theta\}$. For example, $\{\kappa_{\theta} \colon \theta \in
  \Theta\}$ could be the set of Gaussian kernels with bandwidth
  $\theta \in \Theta$ for some bounded set $\Theta \subset
  \mathbb{R}_{+}$. 
\end{remark}

\section{Experimental Results}
\label{sec:experimental-results}
In this section we illustrate our test statistic and procedure with two examples.
The first example investigates the comparison of distinct two-block stochastic blockmodels. 
The second example considers graphs from a protein network dataset and uses our proposed test statistic to build a classifier.

\subsection{Stochastic Blockmodel Example} \label{sec:sbm_example}
We illustrate the hypothesis tests through several simulated and real
data examples. For our first example, let $F_{\epsilon}$ for a given
$\epsilon > 0$ be mixture of point masses corresponding to
a two-block stochastic block model with block membership probabilities
$(0.4, 0.6)$ and block probabilities $\mathbf{B}_{\epsilon} =
\Bigl[ \begin{smallmatrix} 0.5 + \epsilon & 0.2 \\ 0.2 & 0.5 +
  \epsilon \end{smallmatrix}\Bigr]$. 
We then test, for a given $\epsilon > 0$, the hypothesis $H_0 \colon
F_{0} \upVdash F_{\epsilon}$ against the alternative $H_{A} \colon
F_{0} \nupVdash F_{\epsilon}$ using the kernel-based testing procedure
of \S~\ref{sec:main-results}. The kernel is chosen to be the Gaussian
kernel with bandwidth $\sigma = 0.5$. We first evaluate the
performance through simulation using $1000$ Monte Carlo replicates; in
each replicate we sample two graphs on $n$ vertices from
$\mathrm{RDPG}(F_0)$ and one graph on $n$ vertices from
$\mathrm{RPDG}(F_{\epsilon})$. We then perform an adjacency spectral
embedding on the graphs, in which we embed the graphs into
$\mathbb{R}^{2}$, and we proceed to compute the kernel-based test
statistic. 
We evaluate the performance of the test
procedures for both $U_{n,m}(\mathbf{X}, \mathbf{Y})$ and
$U_{n,m}(\hat{\mathbf{X}}, \hat{\mathbf{Y}})$ by estimating the 
  power of the test statistic for various choices of $n \in \{100,
  200, 500, 1000 \}$ and $\epsilon \in \{0.02, 0.05, 0.1\}$ through
  Monte Carlo simulation. The significance level is set to $\alpha =
  0.05$ and the rejection regions are specified via $B = 200$
  bootstrap permutation using either the true latent positions
  $\mathbf{X}$ and $\mathbf{Y}$ or the estimated latent positions $\hat{\mathbf{X}}$ and
  $\hat{\mathbf{Y}}$. These estimates are given in Table~\ref{tab:bootstrap_identity}. 
\begin{table*}[!htbp]
  \footnotesize
  \centering
\begin{tabular}{rrrrrrr}
 & \multicolumn{2}{c}{$\epsilon = 0.02$} &
 \multicolumn{2}{c}{$\epsilon = 0.05$} & \multicolumn{2}{c}{$\epsilon
   = 0.1$} \\ 
 $n$ & $\{\mathbf{X}, \mathbf{Y}\}$ & $\{\hat{\mathbf{X}},
 \hat{\mathbf{Y}}\}$ & $\{\mathbf{X}, \mathbf{Y}\}$ & $\{\hat{\mathbf{X}},
 \hat{\mathbf{Y}}\}$ & $\{\mathbf{X}, \mathbf{Y}\}$ &
 $\{\hat{\mathbf{X}}, \hat{\mathbf{Y}}\}$ \\ \midrule 
 $100$ & $0.07$ & $0.06$ & $0.07$ & $0.09$  & $0.21$ & $0.27$ \\
 $200$ & $0.06$ & $0.09$ & $0.11$ & $0.17$  & $0.89$ & $0.83$ \\
 $500$ & $0.08$ & $0.1$ & $0.37$ & $0.43$  & $1$ & $1$ \\
 $1000$ & $0.1$ & $0.14$ & $1$ & $1$ & $1$ & $1$
\end{tabular}
\caption{Power estimates for testing the null hypothesis $F \upVdash
  G$ at a significance level of $\alpha = 0.05$ using bootstrap
  permutation tests for the $U$-statistics
  $U_{n,m}(\hat{\mathbf{X}}, \hat{\mathbf{Y}})$ and
  $U_{n,m}(\mathbf{X}, \mathbf{Y})$. In each bootstrap test, $B =
  200$ bootstrap samples were generated. Each estimate of power is 
  based on $1000$ Monte Carlo replicates of the corresponding
  bootstrap test.
}
\label{tab:bootstrap_identity}
\end{table*}

\subsection{Classification of protein networks}
For our last example, we show how the statistics
$U_{n,m}(\hat{\mathbf{X}}, \hat{\mathbf{Y}})$ can also be adapted for
use in graphs classification. More concretely, we consider the problem
of classifying proteins network into enzyme versus non-enzymes. We use
the dataset of \citet{dobson03:_distin}, which consists of $1178$
protein networks labeled as enzymes ($691$ networks) and non-enzymes
($487$ networks). For our classification procedure, we first embed
each of the protein networks into $\mathbb{R}^{5}$ using adjacency
spectral embedding. The choice of $d = 5$ is chosen from among the
choices of embedding dimensions ranging from $d = 2$ through $d = 15$
to minimize the classification error rate. We then compute a
$1178 \times 1178$ matrix $\mathbf{S}$ of pairwise dissimilarity between
the adjacency spectral embedding of the protein networks using a
Gaussian kernel with bandwidth $h = 1$. The classifier is a $k$-NN
classifier using the dissimilarities in $\mathbf{S}$ in place of the
Euclidean distance. We evaluate the classification accuracy using a
$10$-fold cross validation. The results are presented in
Table~\ref{tab:enzyme}. For the purpose of comparison, we also include
the accuracy of several other classifiers that were previously applied
on this data set, see
\citet{dobson03:_distin,borgwardt05:_protein}. The results of
\citet{dobson03:_distin} are based on modeling the proteins using
various features such as secondary-structure content, surface
properties, ligands, and amino acid propensities, and then training a
SVM using a radial basis kernel on these feature vectors. The results
of \citet{borgwardt05:_protein} are based on representing the proteins
as graphs, using their secondary-structure content, and then training
a SVM classifier using a random walk kernel on the result graphs.  The
accuracy of our straightforward classifier, which does not use any
information about associated secondary structure, is comparable to
that obtained from using SVM with a well-designed features kernel or
well-designed graph kernels.
\begin{table}[htbp]
  \centering
  \begin{tabular}{ccc}
    Classifier & Accuracy ($\%$) \\ \midrule
    SVM with optimized feature vector kernel \citep{dobson03:_distin} & 80.17 \\
    SVM with random walk kernel with secondary structure \citep{borgwardt05:_protein} & 77.30 \\
    $k$-NN with dissimilarities based on $U_{n,m}$ & 78.20 
  \end{tabular}
  \caption{Classification accuracy on the enzyme dataset. }
  \label{tab:enzyme}
\end{table}

\section{Extensions}
\label{sec:extensions}
In this section we will consider extensions to alternative hypothesis
tests that consider looser notions of equality between the two
distributions.  These notions may be quite useful in practice due to
variations in graph properties that one may want to ignore in a
comparison of the graphs.  We do not formally state results for these
extensions but we note that they can be derived in a similar manner to 
Theorem~\ref{thm:mmd_unbiased_ase}; see 
\S~\ref{sec:scaling_proof} and \S~\ref{sec:projection_proof} in
the appendix.

\subsection{Scaling case}
\label{sec:scaling-case}
We now consider the case of testing the hypothesis that the
distributions $F$ and $G$ are equal up to scaling. 
In particular the test 
  \begin{align*}
\quad H_{0} \colon F \upVdash G \circ c \quad \text{for
    some $c > 0$}
  \quad \text{against} \quad H_{A} \colon F  \nupVdash G \circ c \quad
  \text{for any  $c > 0$},
  \end{align*}
  where $Y \sim F \circ c$ if $cY \sim F$.
The test statistic is now
a simple modification of the one in
Theorem~\ref{thm:mmd_unbiased_ase}, i.e., we first scale the adjacency
spectral embeddings by the norm of the empirical means before
computing the kernel test statistic. 
In particular if we let 
\begin{align*}
  \hat{s}_{X}=n^{-1/2}\|\hat{\mathbf{X}}\|_{F}, \quad \hat{s}_{Y}=m^{-1/2}
    \|\hat{\mathbf{Y}}\|_{F}, \quad 
      s_{X}= n^{-1/2} \|\mathbf{X}\|_{F},\quad s_{Y} =
    m^{-1/2} \|\mathbf{Y}\|_{F},
  \end{align*}
  then the conclusions of Theorem~\ref{thm:mmd_unbiased_ase} hold
  where we use
  $U_{n,m}(\hat{\mathbf{X}}/\hat{s}_{X},\hat{\mathbf{Y}}/\hat{s}_{Y})$
  as the test statistic in comparison to
  $U_{n,m}(\mathbf{X}/s_{X},\mathbf{Y} \mathbf{W}_{n,m}/s_{Y})$.  Note
  that we must restrict $c$ so that $G\circ c$ is still a valid
  distribution for an RDPG.

As an example let $F_{\epsilon}$ be the uniform distribution
on $[\epsilon, 1/\sqrt{2}]^{2}$ where $\epsilon \geq 0$ and
let $G$ be the uniform distribution on $[0,1/\sqrt{3}]^{2}$.  
For a given $\epsilon$, we test
the hypothesis $H_0 \colon F_{\epsilon} \upVdash G \circ c$ for some constant $c
> 0$ against the alternative $H_{A} \colon F_{\epsilon} \nupVdash G \circ c$ for
any constant $c > 0$. The testing procedure is based on the test
statistic $(m+n)
U_{n,m}(\hat{\mathbf{X}}/\hat{s}_X, \hat{\mathbf{Y}}/\hat{s}_Y)$ using
a Gaussian kernel with bandwidth $\sigma = 0.5$. 
Table~\ref{tab:bootstrap_scaling} is the
analogue of Table~\ref{tab:bootstrap_identity} and presents estimates
of the size and power for $U_{n,m}(\mathbf{X}/s_{X}, \mathbf{Y}/s_Y)$ and
$U_{n,m}(\hat{\mathbf{X}}/\hat{s}_X, \hat{\mathbf{Y}}/\hat{s}_Y)$ for various choices of
 $n$ and $\epsilon$. 

\begin{table*}[htbp]
  \footnotesize
  \centering
\begin{tabular}{rrrrrrrrr}
 & \multicolumn{2}{c}{$\epsilon = 0$} & \multicolumn{2}{c}{$\epsilon = 0.05$} &
 \multicolumn{2}{c}{$\epsilon = 0.1$} & \multicolumn{2}{c}{$\epsilon
   = 0.2$} \\ 
 $n$ & $\{\mathbf{X}, \mathbf{Y}\}$ & $\{\hat{\mathbf{X}},
 \hat{\mathbf{Y}}\}$ & $\{\mathbf{X}, \mathbf{Y}\}$ & $\{\hat{\mathbf{X}},
 \hat{\mathbf{Y}}\}$ & $\{\mathbf{X}, \mathbf{Y}\}$ & $\{\hat{\mathbf{X}},
 \hat{\mathbf{Y}}\}$ & $\{\mathbf{X}, \mathbf{Y}\}$ &
 $\{\hat{\mathbf{X}}, \hat{\mathbf{Y}}\}$ \\ \midrule 
 $100$ & $0.05$ & $0.04$ & $0.184$ & $0.02$  & $0.79$ & $0.16$ &1
 & 0.91\\
 $200$ & $0.06$ & $0.1$ & $0.39$ & $0.11$  & $0.98$ & $0.7$ & 1 & 1 \\
 $500$ & $0.07$ & $0.07$ & $0.83$ & $0.66$  & $1$ & $1$ & 1 & 1 \\
 $1000$ & $0.06$ & $0.03$ & $1$ & $0.98$ & $1$ & $1$ & 1 & 1
\end{tabular}
\caption{Power estimates for testing the null hypothesis $F \upVdash G
  \circ c$ at a significance level of $\alpha = 0.05$ using bootstrap
  permutation tests for the $U$-statistics
  $U_{n,m}(\hat{\mathbf{X}}/\hat{s}_X, \hat{\mathbf{Y}}/\hat{s}_Y)$
  and $U_{n,m}(\mathbf{X}/s_X, \mathbf{Y}/s_Y)$. In each bootstrap
  test, $B = 200$ bootstrap samples were generated. Each estimate of
  power is based on $1000$ Monte Carlo replicates of the corresponding
  bootstrap test. The entries for $\epsilon = 0$ coincides with bootstrap
  estimate for the size of the test. }
\label{tab:bootstrap_scaling}
\end{table*}

\subsection{Projection case}
\label{sec:projection-case}
We next consider the case of testing 
\begin{align*}
\quad H_{0} \colon F \circ \pi^{-1} \upVdash G
\circ \pi^{-1} \quad \text{against} \quad H_{A} \colon F \circ \pi^{-1} \nupVdash G \circ \pi^{-1} ,
  \end{align*}
  where $\pi$ is the
  projection $x \mapsto x/\|x\|$ that maps $x$ onto the unit sphere
  in $\mathbb{R}^{d}$. In an abuse of
  notation we will also write $\pi(\mathbf{X})$ to denote the row-wise
  projection of the rows of $\mathbf{X}$ onto the unit sphere. 
  
We shall assume
that $0$ is not an atom of either $F$ or $G$, i.e., $F(0) = G(0) = 0$,
for otherwise the problem is possibly ill-posed: specifically, $\pi(0)$ is
undefined. In addition, for simplicity in the proof, we shall also assume that
the support of $F$ and $G$ is bounded away from $0$, i.e., there
exists some $\epsilon > 0$ such that $F(\{x \colon \|x\| \leq
\epsilon\}) = G(\{x \colon \|x\| \leq \epsilon\}) = 0$. 
A truncation argument with $\epsilon \rightarrow 0$ allows us to handle the general
case of distributions on $\Omega$ where $0$ is not an atom.

To contextualize the test of equality up to projection,
consider the very specific case of the degree-corrected stochastic
blockmodel \citep{karrer2011stochastic}. A degree-corrected stochastic
blockmodel can be view as a random dot product graph whose latent
position $X_v$ for an arbitrary vertex $v$ is of the form $X_v =
\theta_v \nu_{v}$ where $\nu_v$ is sampled from a mixture of point
masses and $\theta_v$ (the degree-correction factor) is sampled from a distribution on
$(0,1]$. Thus, given two degree-corrected stochastic blockmodel
graphs, equality up to projection tests whether the
underlying mixture of point masses (that is, the distribution of the
$\nu_v$) are the same modulo the distribution of the degree-correction
factors $\theta_v$.  

For this test, under the assumption that both $F$ and $G$ have supports
bounded away from the origin, the conclusions of
Theorem~\ref{thm:mmd_unbiased_ase} hold where we use
$U_{n,m}(\pi(\hat{\mathbf{X}}), \pi(\hat{\mathbf{Y}}))$ as the test
statistic and compare it to $U_{n,m}(\pi(\mathbf{X}),\pi(\mathbf{Y})
\mathbf{W}_{n,m})$.

\subsection{Local alternatives and sparsity}
  We now consider the test procedure of
  Theorem~\ref{thm:mmd_unbiased_ase} in the context of (1) local
  alternatives and (2) sparsity.  It is not hard to show that the test
  statistic $U_{n,m}(\hat{\mathbf{X}}, \hat{\mathbf{Y}})$ is also
  consistent against local alternatives, in particular the setting
  $(\mathbf{X}_n, \mathbf{A}_n) \sim \mathrm{RDPG}(F_n)$,
  $(\mathbf{Y}_n, \mathbf{B}_n) \sim \mathrm{RDPG}(G_n)$ with
  $\|\mu[F_n] - \mu[G_n]\|_{\mathcal{H}} \rightarrow 0$. In this
  setting, the accuracy of $\hat{\mathbf{X}}_n$ and $\hat{\mathbf{Y}}_n$
  as estimates for $\mathbf{X}_n$ and $\mathbf{Y}_n$ is unchanged; the
  only difference is that the distance between $F_n$ and $G_n$ is
  shrinking. Thus Eq.~\eqref{eq:conv_mmdXhat_null} and
  Eq.~\eqref{eq:conv_mmdXhat_alt} continue to hold and the test
  procedure is consistent against all local alternatives for which
  $\|\mu[F] - \mu[G] \|_{\mathcal{H}} = \omega(n^{-1/2} \log^{K}(n))$
  for some integer $K \geq 2$
  (c.f. \citet[Theorem~13]{gretton12:_kernel_two_sampl_test}). 

  Another related setting is that of sparsity, in which
  $(\mathbf{X}_n, \mathbf{A}_n) \sim \mathrm{RDPG}(\alpha_n^{1/2} F)$,
  $(\mathbf{Y}_n, \mathbf{B}_n) \sim \mathrm{RDPG}(\alpha_n^{1/2} G)$, with
  $F$ and $G$ being fixed distributions but the sparsity factor
  $\alpha_n \rightarrow 0$. That is to say, $(\mathbf{X}_n,
  \mathbf{A}_n) \sim \mathrm{RDPG}(\alpha_n^{1/2} F)$ for $\alpha_n \leq
  1$ if the rows of
  $\mathbf{X}_n$ are sampled i.i.d from $F$ and, conditioned on
  $\mathbf{X}_n$, $\mathbf{A}_n$ is a random $n \times n$ adjacency
  matrix with probability 
  \begin{equation*} \mathbb{P}[\mathbf{A} | \{X_i\}_{i=1}^{n}] =
    \prod_{i \leq j} (\alpha_n X_i^{\top} X_j)^{\mathbf{A}_{ij}} (1 -
    \alpha_n X_i^{\top}
    X_j)^{1 - \mathbf{A}_{ij}}. 
  \end{equation*} 
  Now the accuracy of $\hat{\mathbf{X}}_n$ and
  $\hat{\mathbf{Y}}_n$ as estimates for $\mathbf{X}_n$ and $\mathbf{Y}_n$
  decreases with $\alpha_n$ due to increasing sparsity. More
  specifically, if $\hat{\mathbf{X}}_n$ denotes the adjacency spectral
  embedding of $\mathbf{A}_n$ where $(\mathbf{X}_n, \mathbf{A}_n) \sim
  \mathrm{RDPG}(\alpha_n F)$, then Lemma~\ref{lem:2} can be extended
  to yield that, with probability at least $1 - 4 \eta$, there exists
  an orthogonal matrix $\mathbf{W}_n$ such that
  \begin{equation}
    \label{eq:14}
   \| \alpha_n^{-1/2} \hat{\mathbf{X}}_n - \mathbf{X}_n \mathbf{W}_n \|_{F} \leq
   \alpha_n^{-1/2} C_1
  \end{equation}
  for some constant $C_1$. 
  We note that there are $n$ rows in
  $\hat{\mathbf{X}}_n$ and hence, on average, we have that for each
  index $i$, $\|\alpha_n^{-1/2} \hat{X}_i - \mathbf{W}_n X_i\| \leq (n
  \alpha_n)^{-1/2} C_1$ with high probability. Thus, if $n \alpha_n
  \rightarrow \infty$, we have that, on average, each 
  $\alpha_n^{-1/2} \hat{X}_i$ is a consistent estimate of the
  corresponding $X_i$. Thus we should expect that there exists some sequence of orthogonal matrices $\mathbf{V}_n$ such that  
  $|U_{n,m}(\alpha_n^{-1/2} \hat{\mathbf{X}}_n, \alpha_n^{-1/2} \hat{\mathbf{Y}}_n) -
  U_{n,m}(\mathbf{X}_n, \mathbf{Y}_n \mathbf{V}_{n})| \rightarrow 0$
  as $n \rightarrow \infty$. More formally, we have the following.
  
 \begin{proposition}
  \label{prop:mmd_unbiased_sparse}
  Let $(\mathbf{X}_n, \mathbf{A}_n) \sim \mathrm{RDPG}(\alpha_n^{1/2} F)$ and
  $(\mathbf{Y}_m, \mathbf{B}_m) \sim \mathrm{RDPG}(\beta_m^{1/2} G)$ be independent
  random dot product graphs with latent position distributions $F$ and
  $G$ and sparsity factor $\alpha_n$ and $\beta_m$, respectively. 
  Furthermore, suppose that both $F$ and $G$ satisfies the
  distinct eigenvalues condition in Assumption~\ref{ass:rank_F} and
  that $\alpha_n$ and $\beta_m$ are known. 
  Consider the hypothesis test
  \begin{align*}
    H_{0} \colon F \upVdash G \quad 
    \text{against} \quad H_{A} \colon F  \nupVdash G.
  \end{align*}
  Denote by $\hat{\mathbf{X}}_n = \{\hat{X}_1, \dots, \hat{X}_n\}$
  and $\hat{\mathbf{Y}}_m = \{\hat{Y}_1, \dots, \hat{Y}_m\}$ the
  adjacency spectral embedding of $\mathbf{A}_n$ and
  $\mathbf{B}_m$, respectively.
  Let $\mathbf{W}_1$
  and $\mathbf{W}_{2}$
  be $d \times d$ orthogonal matrices in the eigendecomposition
  $\mathbf{W}_1 \mathbf{S}_1 \mathbf{W}_{1}^{\top} = \mathbf{X}_n^{\top}
  \mathbf{X}_n$, $\mathbf{W}_{2} \mathbf{S}_{2} \mathbf{W}_{2} =
  \mathbf{Y}_m^{\top} \mathbf{Y}_m$, respectively.  
Suppose that $m, n \rightarrow \infty$, $\tfrac{m}{m+n} \rightarrow \rho \in (0,1)$ and furthermore that $n \alpha_n
  = \omega(\log^{4}{n})$ and $m \beta_m = \omega(\log^{4} m)$. Then 
  the sequence of matrices $\mathbf{W}_{n,m} = \mathbf{W}_{2} \mathbf{W}_{1}^{\top}$ satisfies
  \begin{equation}
    \label{eq:conv_mmdXhat_null_sparse}
    U_{n,m}(\alpha_n^{-1/2} \hat{\mathbf{X}}_n, \beta_m^{-1/2} \hat{\mathbf{Y}}_m) -
    U_{n,m}(\mathbf{X}_n, \mathbf{Y}_m \mathbf{W}_{n,m}) \overset{\mathrm{a.s.}}{\longrightarrow} 0.
  \end{equation}
\end{proposition}
\begin{proof}[Proof Sketch]
  Let $\psi =   U_{n,m}(\alpha_n^{-1/2} \hat{\mathbf{X}}_n, \beta_m^{-1/2} \hat{\mathbf{Y}}_m) - U_{n,m}(\mathbf{X}_n,
  \mathbf{Y}_m \mathbf{W}_{n,m} )$.  
 We have 
 \begin{equation*}
   \begin{split}
\psi & = \frac{1}{n(n-1)}
    \sum_{j \not = i}
    (\kappa( \alpha_n^{-1/2} \hat{X}_i,
     \alpha_n^{-1/2} \hat{X}_j) - \kappa(\mathbf{W}_1 X_i, \mathbf{W}_1 X_j)) \\ &- \frac{2}{mn} \sum_{i=1}^{n}
    \sum_{k=1}^{m} (\kappa(\alpha_n^{-1/2} \hat{X}_i,
    \beta_m^{-1/2} \hat{Y}_k) - \kappa(\mathbf{W}_1 X_i, \mathbf{W}_2 Y_k)) \\ &+ \frac{1}{m(m-1)} \sum_{l \not
      = k} \kappa( \beta_m^{-1/2} \hat{Y}_k,  \beta_m^{-1/2} \hat{Y}_l) -
    \kappa(\mathbf{W}_2 Y_k, \mathbf{W}_2 Y_l)).
   \end{split}
 \end{equation*}
 Let $\mathcal{S}_X \subset \{1,2,\dots,n\}$ and $\mathcal{S}_Y
 \subset \{1,2,\dots,m\}$ be defined by 
 \begin{gather*}
   \mathcal{S}_X = \{i \, \, \colon \,\, \|\alpha_n^{-1/2} \hat{X}_i -
 \mathbf{W}_1 X_i \| \leq C_1 (n \alpha_n)^{-1/2} \log{n} \}, \\ 
   \mathcal{S}_Y = \{k \, \, \colon \,\, \|\beta_m^{-1/2} \hat{Y}_k -
 \mathbf{W}_2 Y_k \| \leq C_2 (m \beta_m)^{-1/2} \log{m} \}.  
 \end{gather*}
From Eq.~\eqref{eq:14}, the number of indices $i \not\in S_X$
is of order $o(n)$, with high probability. Similarly, the number of
indices $k \not \in S_Y$ is of order $o(m)$ with high probability. 
Therefore,
\begin{equation*}
  \begin{split}
\psi & = \frac{1}{n(n-1)}
    \sum_{i \in S_X} \sum_{ j \in S_X}
    (\kappa( \alpha_n^{-1/2} \hat{X}_i,
     \alpha_n^{-1/2} \hat{X}_j) - \kappa(\mathbf{W}_1 X_i,
     \mathbf{W}_1 X_j)) \\ &- \frac{2}{mn} \sum_{i \in S_X}
    \sum_{k \in S_Y} (\kappa(\alpha_n^{-1/2} \hat{X}_i,
    \beta_m^{-1/2} \hat{Y}_k) - \kappa(\mathbf{W}_1 X_i, \mathbf{W}_2
    Y_k)) \\ &+ \frac{1}{m(m-1)} \sum_{k \in S_Y} \sum_{l \in S_Y} \kappa( \beta_m^{-1/2} \hat{Y}_k,  \beta_m^{-1/2} \hat{Y}_l) -
    \kappa(\mathbf{W}_2 Y_k, \mathbf{W}_2 Y_l)) + o(1).
   \end{split}
\end{equation*}
We consider the term $1/(n(n-1)) \sum_{i \in S_X} \sum_{ j \in S_X}
    (\kappa( \alpha_n^{-1/2} \hat{X}_i,
     \alpha_n^{-1/2} \hat{X}_j) - \kappa(\mathbf{W}_1 X_i,
     \mathbf{W}_1 X_j))$.  
By the differentiability of $\kappa$ and compactness of $\Omega$, we
have
$$ |\kappa( \alpha_n^{-1/2} \hat{X}_i,  \alpha_n^{-1/2} \hat{X}_j) - \kappa(\mathbf{W}_1 X_i,
  \mathbf{W}_1 X_j)| \leq C \max \{ \| \alpha_n^{-1/2} \hat{X}_i - \mathbf{W}_1 X_i \|, \|
  \alpha_n^{-1/2} \hat{X}_j - \mathbf{W}_1 X_j \| \} $$
for some constant $C$ independent of $i$ and $j$. 
Thus, \begin{equation*}
  \Bigl|\frac{1}{n(n-1)} \sum_{i \in S_X} \sum_{j \in S_X}     (\kappa( \alpha_n^{-1/2} \hat{X}_i,
     \alpha_n^{-1/2} \hat{X}_j) - \kappa(\mathbf{W}_1 X_i,
     \mathbf{W}_n X_j)) \Bigr| \leq 
  \max_{i \in S_X} C \|\alpha_n^{-1/2} \hat{X}_i -
  \mathbf{W}_1 X_i \|.
\end{equation*}
Similar reasoning yields
\begin{equation*}
  \begin{split}
  |\psi| & \leq 2 C (\max_{i \in S_X} \|\alpha_n^{-1/2} \hat{X}_i -
  \mathbf{W}_1 X_i \| + \max_{k \in S_Y} \|\beta_m^{-1/2}
  \hat{Y}_k - \mathbf{W}_2 Y_k \|) + o(1) \\ & \leq 2 C (C_1 (n
  \alpha_n)^{-1/2} \log n + C_2 (m \beta_m)^{-1/2} \log m) + o(1) 
  \end{split}
\end{equation*}
with high probability. 
As $n \alpha_n = \omega(\log^{4} n)$ and $m \beta_m = \omega(\log^{4}
m)$, we have
$\psi \overset{\mathrm{a.s.}}{\longrightarrow} 0$ as $m, n \rightarrow \infty$. 
\end{proof}

We assume in the statement of Proposition~\ref{prop:mmd_unbiased_sparse} that the sparsity factors
$\alpha_n$ and $\beta_m$ are known. If $\alpha_n$ and $\beta_m$ are
unknown, then they can be estimated from the adjacency spectral
embedding of $\mathbf{A}_n$ and $\mathbf{B}_m$, but only up to some
constant factor. Hence the hypothesis test of
\begin{align*}
  H_{0} \colon F \upVdash G \quad 
  \text{against} \quad H_{A} \colon F  \nupVdash G
\end{align*}
is no longer meaningful (as the sparsity factors $\alpha_n$ and
$\beta_m$ cannot be determined uniquely) and one should consider
instead the hypothesis test of equality up to scaling of
\S~\ref{sec:scaling-case}.  
 
We note that the conclusion of
Proposition~\ref{prop:mmd_unbiased_sparse}, namely
Eq.~\eqref{eq:conv_mmdXhat_null_sparse}, is weaker than that of
Theorem~\ref{thm:mmd_unbiased_ase} due to the lack of the $m+n$ factor
in Eq.~\eqref{eq:conv_mmdXhat_null_sparse} as compared to
Eq.~\eqref{eq:conv_mmdXhat_null} and Eq.~\eqref{eq:conv_mmdXhat_alt}. 
The more difficult question, and one which we will not address in this paper, is to
refine the rate of convergence to zero of
Eq.~\eqref{eq:conv_mmdXhat_null_sparse} in the sparse setting. We suspect, however, that
Eq.~\eqref{eq:conv_mmdXhat_null} will not hold in the case where
$\alpha_n = o(n^{-1/2})$ and $\beta_m = o(m^{-1/2})$. 
Nevertheless, Proposition~\ref{prop:mmd_unbiased_sparse} still yields
yields a test procedure that is
consistent against any alternative, provided that both $F$ and $G$
satisfy Assumption~\ref{ass:rank_F}, and that the sparsity factors
$\alpha_n$ and $\beta_m$ do not converge to $0$ too quickly.

\section{Discussion}
\label{sec:conclusions}
In summary, we show in this paper that the adjacency spectral
embedding can be used to generate simple and intuitive test statistics
for the nonparametric inference problem of testing whether two random
dot product graphs have the same or related distribution of latent
positions. The two-sample formulations presented here and the
corresponding test statistics are intimately related. Indeed, for
random dot product graphs, the adjacency spectral embedding yields a
consistent estimate of the latent positions as points in
$\mathbb{R}^{d}$; there then exist a wide variety of classical and
well-studied testing procedures for data in Euclidean spaces.

New results on stochastic blockmodels suggest that they can be
regarded as a universal approximation to graphons in exchangeable
random graphs, see e.g., \citet{yang14:_nonpar,wolfe13:_nonpar}.  There
is thus potential theoretical value in the formulation of two-sample
hypothesis testing for latent position models in terms of a random dot
product graph model on $\mathbb{R}^{d}$ with possibly varying
$d$. However, because the link function and the distribution of latent
positions are intertwined in the context of latent position graphs,
any proposed test procedure that is sufficiently general might also
possess little to no power.

The two-sample hypothesis testing we consider here is also closely
related to the problem of testing goodness-of-fit; the results in this
paper can be easily adapted to address the latter question. In
particular, we can test, for a given graph, whether the graph is
generated from some specified stochastic blockmodel. A more general
problem is that of testing whether a given graph is generated
according to a latent position model with a specific link
function. This problem has been recently studied; see
\citet{yang14:_nonpar} for a brief discussion, but much remains to be
investigated. For example, the limiting distribution of the test
statistic in \citet{yang14:_nonpar} is not known.

Finally, two-sample hypothesis testing is also closely related to
testing for independence; given a random sample $\{(X_i,Y_i)\}$ with
joint distribution $F_{XY}$ and marginal distributions $F_X$ and
$F_Y$, $X$ and $Y$ are independent if the $F_{XY}$ differs from the
product $F_XF_Y$. For example, the Hilbert-Schmidt independence
criterion is a measure for statistical dependence in terms of the
Hilbert-Schmidt norm of a cross-covariance operator. It is based on
the maximum mean discrepancy between $F_{XY}$ and $F_{X}
F_{Y}$. Another example is Brownian distance covariance of
\citet{szekely09:_brown}, a measure of dependence based on the energy
distance between $F_{XY}$ and $F_{X} F_{Y}$. In particular, consider
the test of whether two given two random dot product graphs
$(\mathbf{X}, \mathbf{A}) \sim \mathrm{RDPG}(F_X)$ and $(\mathbf{Y},
\mathbf{B}) \sim \mathrm{RDPG}(F_Y)$ on the same vertex set have
independent latent position distributions $F_X$ and $F_Y$. While we
surmise that it may be possible to adapt our present results to this
question, we stress that the conditional independence of $\mathbf{A}$
given $\mathbf{X}$ and of $\mathbf{B}$ given $\mathbf{Y}$ suggests
that independence testing may merit a more intricate approach.
\bibliography{../biblio}
\appendix
\section{Additional Proofs}
\label{sec:proofs}
\label{sec:append-addit-lemm}

{\bf Proof of Lemma~\ref{lem:emp_proc}}:
As $\kappa$ is twice continuously differentiable, $\mathcal{F}$ is also
twice continuously differentiable \citep[Corollary~4.36]{steinwart08:_suppor_vector_machin}.
Denote by $\mathbf{W}$ the orthogonal matrix such that $\mathbf{X} =
\mathbf{U}_{\mathbf{P}} \mathbf{S}_{\mathbf{P}}^{1/2} \mathbf{W}$.

Let $f \in \mathcal{F}_{\Phi}$, a Taylor expansion of $f$ then yields
  \begin{equation*}
    \begin{split}
     \frac{1}{\sqrt{n}} \sum_{i=1}^{n} \Bigl(f(\mathbf{W} \hat{X}_i) -
    f(X_i)\Bigr)
     &= \frac{1}{\sqrt{n}} \sum_{i=1}^{n} (\partial
     f)(X_i)^{\top}(\mathbf{W} \hat{X}_i - X_i) \\ & +
     \frac{1}{2 \sqrt{n}} \sum_{i} (\mathbf{W} \hat{X}_i - X_i)^{\top}
    (\partial^{2} f)(X_i^{*}) (\mathbf{W} \hat{X}_i - X_i)
    \end{split}
  \end{equation*}
  where, for any $i$, $X_i^{*} \in \mathbb{R}^{d}$ is such that
  $\|X_i^{*} - X_i \| \leq \|\mathbf{W} \hat{X}_i - X_i\|$. We first
  bound the quadratic terms, i.e. those depending on $\partial^{2}
  f$. We note that since $\kappa$ is twice continuously
  differentiable, $\sup_{f \in \mathcal{F}_\Phi, X \in \Omega}
  \|(\partial^{2} f)(X) \|$ is bounded (the norm under consideration
  is the spectral norm on matrices). Therefore, 
  \begin{equation*} \begin{split} \sup_{f \in \mathcal{F}_{\Phi}}
      \Big| \sum_{i=1}^{n} \frac{(\mathbf{W} \hat{X}_i - X_i)^{\top}
        (\partial^{2} f)(X_i^{*}) (\mathbf{W} \hat{X}_i -
        X_i)}{\sqrt{n}} \Big|&
      \leq \sup_{f \in \mathcal{F}, X \in \Omega} \frac{\|
        (\partial^{2} f)(X) \| \| \hat{\mathbf{X}} \mathbf{W} -
      \mathbf{X}\|_{F}^{2} } {\sqrt{n}}.
    \end{split}
  \end{equation*}
Hence, by applying Lemma~\ref{lem:2} to bound
  $\| \hat{\mathbf{X}} \mathbf{W} - \mathbf{X}\|_{F}^{2}$ in the above
  expression, we have
  \begin{equation*}
    \sup_{f \in \mathcal{F}_{\Phi}}
      \Big| \sum_{i=1}^{n} \frac{(\mathbf{W} \hat{X}_i - X_i)^{\top}
        (\partial^{2} f)(X_i^{*}) (\mathbf{W} \hat{X}_i -
        X_i)}{\sqrt{n}} \Big| \overset{\mathrm{a.s.}}{\longrightarrow} 0
  \end{equation*}
  as $n \rightarrow \infty$. 

  We now bound the linear terms, i.e., those depending on $\partial
  f$. For any $f \in \mathcal{F}_\Phi$, and any $X_1, \dots, X_n$, let
  $\mathbf{M}(\partial f)=\mathbf{M}(\partial f; X_1, \cdots, X_n) \in \mathbb{R}^{n
    \times d}$ be the matrix whose rows are the vectors
  $(\partial f)(X_i)$. 
  We then have
  \begin{equation*}
    \begin{split} 
      \zeta(f) &:= \frac{1}{\sqrt{n}} \sum_{i=1}^{n} (\partial
      f)(X_i)^{\top} (\mathbf{W} \hat{X}_i - X_i) \\ &= \frac{1}{\sqrt{n}} \mathrm{tr}
      \Bigl(( \hat{\mathbf{X}} \mathbf{W} - \mathbf{X})
      [\mathbf{M}(\partial f)]^{\top}\Bigr) = \frac{1}{\sqrt{n}} \mathrm{tr} 
      \Bigl((\mathbf{U}_{\mathbf{A}} \mathbf{S}_{\mathbf{A}}^{1/2} -
      \mathbf{U}_{\mathbf{P}} \mathbf{S}_{\mathbf{P}}^{1/2})
      \mathbf{W} [\mathbf{M}(\partial f)]^{\top}\Bigr). 
    \end{split}
  \end{equation*} 
  Now
  $\mathbf{A} = \mathbf{U}_{\mathbf{A}} \mathbf{S}_{\mathbf{A}}
  \mathbf{U}_{\mathbf{A}} ^{\top} + \mathbf{E}$
  where, as we recall  in Definition~\ref{def:ase},
  $\mathbf{S}_{\mathbf{A}}$ is the diagonal matrix containing the $d$
  largest eigenvalues of $|\mathbf{A}|$ (which coincides, with high probability, with the eigenvalues of $A$) and $\mathbf{U}_{\mathbf{A}}$
  is the matrix whose columns are the corresponding
  eigenvectors. The eigendecomposition of $\mathbf{E}$
  can be written in terms of the eigenvalues and eigenvectors that are
  not included in $\mathbf{S}_{\mathbf{A}}$ and
  $\mathbf{U}_{\mathbf{A}}$. Thus
  $\mathbf{E} \mathbf{U}_{\mathbf{A}} = \bm{0}$ and
  $\mathbf{U}_{\mathbf{A}} \mathbf{S}_{\mathbf{A}}^{1/2} =
  \mathbf{U}_{\mathbf{A}} \mathbf{S}_{\mathbf{A}}
  \mathbf{U}_{\mathbf{A}}^{\top} \mathbf{U}_{\mathbf{A}}
  \mathbf{S}_{\mathbf{A}}^{-1/2} = (\mathbf{U}_{\mathbf{A}}
  \mathbf{S}_{\mathbf{A}} \mathbf{U}_{\mathbf{A}}^{\top} +
  \mathbf{E}) \mathbf{U}_{\mathbf{A}} \mathbf{S}_{\mathbf{A}}^{-1/2} =
  \mathbf{A} \mathbf{U}_{\mathbf{A}}
  \mathbf{S}_{\mathbf{A}}^{-1/2}$.
  Similarly, $\mathbf{P} = \mathbf{U}_{\mathbf{P}}
  \mathbf{S}_{\mathbf{P}} \mathbf{U}_{\mathbf{P}} ^{\top}$ (because $\mathbf{P}$ is rank $d$)
  and 
  $\mathbf{U}_{\mathbf{P}} \mathbf{S}_{\mathbf{P}}^{1/2} = \mathbf{P}
  \mathbf{U}_{\mathbf{P}} \mathbf{S}_{\mathbf{P}}^{-1/2}$. Thus, 
  \begin{equation*}
    \begin{split}
    \zeta(f) &= \frac{1}{\sqrt{n}} \mathrm{tr} \Bigl(\mathbf{A}
    \mathbf{U}_{\mathbf{A}} \mathbf{S}_{\mathbf{A}}^{-1/2} - \mathbf{P}
    \mathbf{U}_{\mathbf{P}} \mathbf{S}_{\mathbf{P}}^{-1/2} \Bigr)
    \mathbf{W} [\mathbf{M}(\partial f)]^{\top}\Bigr) \\ &=
    \frac{1}{\sqrt{n}} \mathrm{tr} \Bigl(\Bigl(\mathbf{A}
      (\mathbf{U}_{\mathbf{A}} -
      \mathbf{U}_{\mathbf{P}}) \mathbf{S}_{\mathbf{A}}^{-1/2} + \mathbf{A}
      \mathbf{U}_{\mathbf{P}} (\mathbf{S}_{\mathbf{A}}^{-1/2} -
      \mathbf{S}_{\mathbf{P}}^{-1/2}) + (\mathbf{A} -
      \mathbf{P}) \mathbf{U}_{\mathbf{P}}
      \mathbf{S}_{\mathbf{P}}^{-1/2}\Bigr) \mathbf{W}
      [\mathbf{M}(\partial f)]^{\top}\Bigr).
    \end{split}
  \end{equation*}
  We therefore have
  \begin{equation}
  \label{eq:sup_f_decomposition1}
    \begin{split}
    \sup_{f \in \mathcal{F}_\Phi} |\zeta(f)| &\leq \frac{\sup_{f \in
        \mathcal{F}_\Phi}\|\mathbf{M}(\partial f)\|_{F}}{\sqrt{n}} \Bigl(
    \|\mathbf{A}(\mathbf{U}_{\mathbf{A}} - \mathbf{U}_{\mathbf{P}})
    \mathbf{S}_{\mathbf{A}}^{-1/2} \mathbf{W} \|_{F}  + \|\mathbf{A}
    \mathbf{U}_{\mathbf{P}} (\mathbf{S}_{\mathbf{A}}^{-1/2} -
    \mathbf{S}_{\mathbf{P}}^{-1/2}) \mathbf{W} \|_{F}\Bigr) \\ &+
    \frac{1}{\sqrt{n}} \,\, \sup_{f \in \mathcal{F}_\Phi}  |\mathrm{tr}
    \Bigl([\mathbf{M}(\partial f)]^{T} (\mathbf{A} - \mathbf{P}) \mathbf{U}_{\mathbf{P}}
    \mathbf{S}_{\mathbf{P}}^{-1/2} \mathbf{W}\Bigr)|.
    \end{split}
  \end{equation}
  We bound the first two terms on the right hand side of Eq.~\eqref{eq:sup_f_decomposition1}
  using the following result from \citet{lyzinski13:_perfec}.
  \begin{lemma}
    \label{lem:4}
    Let $(\mathbf{X}, \mathbf{A}) \sim \mathrm{RDPG}(F)$ and let $c>0$ be
    arbitrary but fixed. There exists $n_0(c)$ such that if $n>n_0$ and
    $\eta$ satisfies $n^{-c} < \eta< 1/2$, then with probability at
    least $1-2 \eta$, the following bounds hold simultaneously
    \begin{gather}
      \label{eq:17}
      \|\mathbf{A} (\mathbf{U}_{\mathbf{A}} - \mathbf{U}_{\mathbf{P}})
      \mathbf{S}_{\mathbf{A}}^{-1/2} \|_{F} 
      \leq \frac{24 \sqrt{2} d
        \log{(n/\eta)}}{\sqrt{\gamma^{5}(F) n}}, \\
      \label{eq:18}
      \|\mathbf{A} \mathbf{U}_{\mathbf{P}} (\mathbf{S}_{\mathbf{A}}^{-1/2}
      - \mathbf{S}_{\mathbf{P}}^{-1/2}) \|_{F} 
      \leq \frac{48 d \log{(n/\eta)}}{\sqrt{\gamma^{7}(F) n}},
    \end{gather}
    where $\gamma(F)$ is the minimum gap between the distinct eigenvalues of the matrix 
    $\mathbb{E}[X_1 X_1^{\top}]$ with $X_1 \sim F$.
  \end{lemma}
  Eq.~\eqref{eq:17} in the above lemma is a restatement of Lemma~3.4
  in \citet{lyzinski13:_perfec} where we have used the fact that the
  maximum row sum of $\mathbf{A}$ is $n$. Eq.~\eqref{eq:18} follows
  from Lemma~3.2 in \citet{lyzinski13:_perfec} and the fact that
  $\|\mathbf{M}_1 \mathbf{M}_2 \|_{2 \to \infty} \leq \|\mathbf{M}_1
  \|_{2 \to \infty} \|\mathbf{M}_2 \|$ for any matrices $\mathbf{M}_1$
  and $\mathbf{M}_2$.  As the individual bound in Eq.~\eqref{eq:17}
  and Eq.~\eqref{eq:18} holds with probabilty $1 - \eta$, they hold
  simultaneously with probability $1 - 2\eta$.

 Lemma~\ref{lem:4} then yields
  \begin{equation*}
\frac{\sup_{f \in
        \mathcal{F}_{\Phi}}\|\mathbf{M}(\partial f)\|_{F}}{\sqrt{n}} \Bigl(
    \|\mathbf{A}(\mathbf{U}_{\mathbf{A}} - \mathbf{U}_{\mathbf{P}})
    \mathbf{S}_{\mathbf{A}}^{-1/2} \mathbf{W} \|_{F}  + \|\mathbf{A}
    \mathbf{U}_{\mathbf{P}} (\mathbf{S}_{\mathbf{A}}^{-1/2} -
    \mathbf{S}_{\mathbf{P}}^{-1/2}) \mathbf{W} \|_{F}\Bigr) 
    \leq \frac{C(F) \log{n}}{\sqrt{n}}
  \end{equation*}
  with probability at least $1 - n^{-2}$, where $C(F)$ is a constant
  depending only on $F$. 

  
  We next show that the last term on the right hand side of
  Eq.~\eqref{eq:sup_f_decomposition1} is also of order
  $n^{-1/2} (\log{n})$ with probability at least $1 - n^{-2}$.
  To control this supremum, we use a chaining argument. Denote by
  $\partial \mathcal{F}_{\Phi}$ the space of functions $\partial \mathcal{F}_{\Phi}
  = \{ \partial f \colon f \in \mathcal{F}_{\Phi}\}$ from $\mathbb{R}^{d}$ to
  $\mathbb{R}^{d}$.  For a given $\partial
  f \in \partial \mathcal{F}_{\Phi}$ let $\|\partial f\|_{\infty}$ denote the
  quantity $\sup_{X \in \Omega} \|(\partial f)(X)\|_{2}$, where
  $\|\cdot\|_{2}$ is the Euclidean norm in
  $\mathbb{R}^{d}$. Similarly, for given $\partial f, \partial g \in
  \mathcal{\partial F}_{\Phi}$, let $\|\partial f - \partial g\|_{\infty}$
  denote $\sup_{X \in \Omega} \|(\partial f - \partial g)(X)\|_{2}$. As
  $\kappa$ is twice continuously differentiable and $\Omega$ is
  compact, $\mathcal{\partial F}_{\Phi}$ is totally bounded with respect to
  $\|\cdot\|_{\infty}$. Put $\delta=\sup_{\partial f \in
    \mathcal{\partial F}_{\Phi}}
  \|\partial f\|_{\infty}$.  Then for any $j \in \mathbb{N}$, we can find a
  finite subset $S_j = \{\partial f_1, \partial f_2, \dots, \partial
  f_{n_j}\}$ of $\partial \mathcal{F}_{\Phi}$
  such that for any $\partial f \in \partial \mathcal{F}_{\Phi}$, there
  exists a $\partial f_{l} \in
  S_j$ with $\|\partial f - \partial f_l\|_{\infty} \leq \delta_{j} := 2^{-j} \delta$.
  We shall assume that $S_j$ is {\em minimal} among all sets with the
  above property. 
 
  Given $S_j$, define $\Pi_{j}$ as the mapping that maps any $\partial
  f \in \partial \mathcal{F}_{\Phi}$ to an (arbitrary) $\partial f_{l}
  \in \mathcal{\partial F}_{\Phi}$ satisfying the condition
  $\|\partial f_l - \partial f \|_{\infty} \leq \delta_{j}$. Denote by
  $\tilde{X}_1, \dots, \tilde{X}_n$ the rows of the matrix $\mathbf{A}
  \mathbf{U}_{\mathbf{P}} \mathbf{S}_{\mathbf{P}}^{-1/2} \mathbf{W}
  $. Then by the separability of $\partial \mathcal{F}_{\Phi}$, we have
  \begin{equation*}
    \begin{split}
      \tilde{\zeta}(f) & := 
     \frac{1}{\sqrt{n}} \sup_{f \in \mathcal{F}_{\Phi}}
     |\mathrm{tr}[\mathbf{M}(\partial f)]^{\top} (\mathbf{A} -
     \mathbf{P}) \mathbf{U}_{\mathbf{P}}
     \mathbf{S}_{\mathbf{P}}^{-1/2} \mathbf{W}| \\ &
     = \sup_{f \in \mathcal{F}_{\Phi}} \Bigl|\frac{1}{\sqrt{n}} 
     \sum_{i=1}^{n} (\partial f)(X_i)^{\top} (\tilde{X}_i - X_i)\Bigr| \\ &
     = \sup_{f \in\mathcal{F}_{\Phi}} \Bigl|\Bigl( \frac{1}{\sqrt{n}} \sum_{i=1}^{n}
    \sum_{j=0}^{\infty} (\Pi_{j+1} \partial f - \Pi_{j} \partial f)(X_i)^{\top}
    (\tilde{X}_i - X_i)\Bigr) + \frac{c_0}{\sqrt{n}}\Bigr|  \\
    &= \sup_{f \in \mathcal{F}_{\Phi}} \Bigl|\Bigl(\frac{1}{\sqrt{n}} 
    \sum_{j=0}^{\infty} \sum_{i=1}^{n} (\Pi_{j+1} \partial f - \Pi_{j}
    \partial f)(X_i)^{\top} (\tilde{X}_i - X_i)\Bigr) + \frac{c_0}{\sqrt{n}}\Bigr| \\
    & \leq \sum_{j=0}^{\infty} \,\sup_{f \in \mathcal{F}_{\Phi}} \, \Bigl |\frac{1}{\sqrt{n}}
    \sum_{i=1}^{n} (\Pi_{j+1} \partial f - \Pi_{j} \partial f)(X_i)^{\top} (\tilde{X}_i -
    X_i) \Bigr| + \Bigl|\frac{c_0}{\sqrt{n}}\Bigr|
    \end{split}
  \end{equation*}
  where $c_0 = \sum_{i=1}^{n} (\Pi_0 \partial f)(X_i)^{T} (\tilde{X}_i - X_i)$.

  The term $n^{-1/2} \sum_{i=1}^{n} (\Pi_{j+1} \partial f -
  \Pi_{j} \partial f)(X_i)^{\top} (\tilde{X}_i -
  X_i)$ can be written as sum of quadratic form, i.e.,
  \begin{equation}
    \label{eq:19}
    \frac{1}{\sqrt{n}} \sum_{i=1}^{n} (\Pi_{j+1} \partial f -
    \Pi_{j} \partial f)(X_i)^{\top} (\tilde{X}_i -
    X_i) = \frac{1}{\sqrt{n}} \sum_{s=1}^{d}
    (\bm{\pi}^{(j,j+1)}_{s}(\partial f))^{\top} (\mathbf{A} -
    \mathbf{P}) \bm{u}_{s} \lambda_{s}^{-1/2}
    \end{equation}
    where $\bm{\pi}^{(j,j+1)}_{s}(\partial f)$ for $s = 1,2,\dots, d$ are the
    columns of the $n \times d$ matrix with rows  
    $\mathbf{W} (\Pi_{j+1} \partial f - \Pi_{j} \partial f)(X_i)$ for $i=1,\dots,n$ and
    $\bm{u}_{s}$ and $\lambda_{s}$ are the eigenvectors and corresponding eigenvalues of
    $\mathbf{P}$. 

    Now, for any vectors $\bm{b} = (b_1, b_2, \dots, b_n)$ and
    $\bm{c} = (c_1, c_2, \dots, c_n)$, 
  \begin{equation*}
    \bm{b}^{T} (\mathbf{A} - \mathbf{P}) \bm{c} = 2 \sum_{i < j}
    b_i (\mathbf{A} - \mathbf{P})_{ij} c_j + \sum_{i}
    \mathbf{P}_{ii} b_i c_i.
  \end{equation*}
  The sum over the indices $i < j$ in the above display is a sum of
  independent random variables. Therefore,  
  Hoeffding's inequality ensures that
  \begin{equation*}
    \begin{split}
    \mathbb{P}[ | 2 \sum_{i < j} b_i  (\mathbf{A} - \mathbf{P}) c_j|
    \geq t]  \leq 2 \exp \Bigl(-\frac{t^2}{  8 \sum_{i < j} b_i^{2}
      c_j^2}\Bigr) \leq 2 \exp \Bigl(-\frac{t^2}{ 8 \|\bm{b}\|^2 \|\bm{c}\|^2}\Bigr).
    \end{split}
  \end{equation*}
  In addition, $\sum_{i} \mathbf{P}_{ii} b_i c_i \leq
  \|\bm{b}\|\|\bm{c}\|$. 
  We apply the above argument to Eq.~\eqref{eq:19}. First,
  $\|\bm{\pi}_{s}^{(j,j+1)}(\partial f)\|_{2} \leq 3/2 \delta_{j} \sqrt{n}$ for
  all $\partial f \in \partial \mathcal{F}$. In addition, $\|\bm{u}_s\| = 1$ for all
  $s$. Hence, for all $t \geq 2 \delta_{j} \lambda_{d}^{-1/2}$,
  \begin{equation*}
    \begin{split}
     \mathbb{P}\Bigl[\frac{1}{\sqrt{n}}\Bigl|
    \sum_{s=1}^{d} (\bm{\pi}^{(j,j+1)}_{s}(\partial f))^{T} (\mathbf{A} -
    \mathbf{P}) \bm{u}_{s} \lambda_{s}^{-1/2}\Bigr| \geq dt\Bigr] \leq 2d
    \exp\Bigl(- \frac{t^2}{ K \delta_{j}^{2}  \lambda_{d}^{-1}}\Bigr)
    \end{split}
  \end{equation*}
  for some universal constant $K > 0$. Let $N_j$ be the cardinality
  of $\{\Pi_{j+1} \partial f - \Pi_{j} \partial f \colon f \in
  \mathcal{F}_{\Phi} \}$. Then by the union bound,
  \begin{equation*}
    \begin{split}
     \mathbb{P}\Bigl[\sup_{f \in \mathcal{F}_{\Phi}}
       \, \Bigl|\frac{1}{\sqrt{n}}
    \sum_{i=1}^{n} (\Pi_{j+1} \partial f - \Pi_{j} \partial f)(X_i)^{T} (\tilde{X}_i -
    X_i)\Bigr| \geq dt\Bigr] \leq 2d N_j
    \exp\Bigl(- \frac{t^2}{ K \delta_{j}^{2} \lambda_{d}^{-1}}\Bigr) .
    \end{split}
  \end{equation*}
  Now $N_j \leq |S_{j+1}|^{2}$ and hence, for any $t_j > 0$, 
  \begin{equation}
    \label{eq:21}
    \begin{split}
    \mathbb{P}\Bigl[\sup_{f \in
      \mathcal{F}_{\Phi}} \, \Bigl|\frac{1}{\sqrt{n}} 
    \sum_{i=1}^{n} (\Pi_{j+1} \partial f - \Pi_{j} \partial f)(X_i)^{T} (\tilde{X}_i - X_i)
    \Bigr| \geq \eta_j \Bigr] &
    \leq 2d \exp(-t_{j}^2). 
  \end{split}
  \end{equation}
  where $\eta_j = d\sqrt{K \delta_{j}^{2} \lambda_{d}^{-1} 
      (t_j^{2} + \log |S_{j+1}|^{2})}$.
  Summing Eq.~\eqref{eq:21} over all $j \geq 0$, and bounding
  $n^{-1/2} c_0$ using another application of Hoeffding's inequality, we arrive at
  \begin{equation*}
    \mathbb{P}[ \sup_{f \in \mathcal{F}_{\Phi}} \Bigl| \tilde{\zeta}(f) \Bigr| \geq \sum_{j=0}^{\infty} K' \eta_j \Bigr] \leq 2d \sum_{j=0}^{\infty} \exp(-t_{j}^2) 
  \end{equation*}
  for some constant $K' > 0$.  We now bound 
  $\sum_{j=0}^{\infty} \eta_j = \sum_{j=0}^{\infty} d\sqrt{K \delta_{j}^{2} \lambda_{d}^{-1}
    (t_j^{2} + \log |S_{j+1}|^{2})}$. To bound $|S_j|$, we
  use the covering number for $\Omega$, i.e.,
  $|S_{j}| \leq (L/\delta_j)^{d}$ \citep[Lemma~2.5]{geer00:_empir_m} for some constant $L$
  independent of $\delta_j$. Then by taking $t_{j}^{2} = 2 (\log{j} + \log{n})$,
  \begin{equation}
    \label{eq:22}
    \mathbb{P}\Bigl[ \,\, \sup_{f \in \mathcal{F}_{\Phi}} \,\, \Bigl|
    \tilde{\zeta}(f) \Bigr| \geq  d \lambda_{d}^{-1/2} (C_1
    \log{n} + C_2) \Bigr] \leq \frac{2dC_3}{n^2}
  \end{equation}
  for some constants $C_1$, $C_2$ and $C_3$. Eq.~\eqref{eq:22} and
  Eq.~\eqref{eq:sup_f_decomposition1} implies
  \begin{equation}
    \label{eq:supBnd}
    \sup_{f \in \mathcal{F}_{\Phi}} |\zeta(f)| \leq  \frac{C(F) 
      \log{n}}{\sqrt{n}} +  d \lambda_{d}^{-1/2} (C_1
    \log{n} + C_2)
  \end{equation}
  with probability at least $1 - (1 + 2dC_3)n^{-2}$. Since there
  exists some constant $c>0$ for which $\lambda_{d}/(cn) \rightarrow
  1$ almost surely, an application of the Borel-Cantelli lemma to
  Eq.~\eqref{eq:supBnd} yields $\sup_{f \in \mathcal{F}_{\Phi}} |\zeta(f)|
  \rightarrow 0$ almost surely. Lemma~\ref{lem:emp_proc} is thus
  established. 
  
\subsection{Proof for the Scaling Case \S\ref{sec:scaling-case}}
\label{sec:scaling_proof}
 The proof parallels that of
 Theorem~\ref{thm:mmd_unbiased_ase}. We sketch here the requisite
 modifications for the case when the null hypothesis $F \upVdash G \circ c$ holds.
 Namely, we show that when $F \upVdash G \circ c$ for some constant $c > 0$,
 \begin{equation}
   \label{eq:6}
   (m+n)(V_{n,m}(\hat{\mathbf{X}}/\hat{s}_X, \hat{\mathbf{Y}}/\hat{s}_Y) -
   V_{n,m}(\mathbf{X}/s_X, \mathbf{Y} \mathbf{W}_{n,m}/s_{Y})
   \overset{\mathrm{a.s}}{\longrightarrow} 0. 
 \end{equation}
Define $\xi_{W}, \hat{\xi} \in \mathcal{H}$ by 
 \begin{gather*}
   \xi_{W} = \frac{\sqrt{m+n}}{n} \sum_{i=1}^{n} \kappa(\mathbf{W}_1 X_i/s_X, \cdot) -
   \frac{\sqrt{m+n}}{m} \sum_{k=1}^{m} \kappa(\mathbf{W}_{2} Y_k/s_{Y}, \cdot), \\
   \hat{\xi} = \frac{\sqrt{m+n}}{n} \sum_{i=1}^{n} \kappa(\hat{X}_i/\hat{s}_X, \cdot) -
   \frac{\sqrt{m+n}}{m} \sum_{k=1}^{m} \kappa(\hat{Y}_k/\hat{s}_{Y}, \cdot).
 \end{gather*}
 Define $r_1$ and $r_2$ similar to that in the proof of
 Theorem~2, i.e., 
 \begin{gather*}
   r_1 = \frac{m+n}{n(n-1)} \sum_{i=1}^{n} \Bigl\{
   \kappa\Bigl(\tfrac{\hat{X}_i}{\hat{s}_X},
   \tfrac{\hat{X}_i}{\hat{s}_X}\Bigr) -
   \kappa\Bigl(\tfrac{X_i}{s_X}, \tfrac{X_i}{s_X}\Bigr)\Bigr\}
       + \frac{m+n}{m(m-1)} \sum_{k=1}^{m} \Bigl\{
   \kappa\Bigl(\tfrac{\hat{Y}_k}{\hat{s}_Y},
   \tfrac{\hat{Y}_k}{\hat{s}_Y}\Bigr) -
   \kappa\Bigl(\tfrac{Y_k}{s_Y}, \tfrac{Y_k}{s_Y}\Bigr)\Bigr\}, \\
   r_2 = \frac{m+n}{n^{2}(n-1)} \sum_{i=1}^{n} \sum_{j=1}^{n}
   \Bigl\{
   \kappa\Bigl(\tfrac{\hat{X}_i}{\hat{s}_X},
   \tfrac{\hat{X}_j}{\hat{s}_X}\Bigr) -
   \kappa\Bigl(\tfrac{X_i}{s_X}, \tfrac{X_j}{s_X}\Bigr)\Bigr\}
    + \frac{m+n}{m^2(m-1)}
   \sum_{k=1}^{m} \sum_{l=1}^{m} \Bigl\{
   \kappa\Bigl(\tfrac{\hat{Y}_k}{\hat{s}_Y},
   \tfrac{\hat{Y}_l}{\hat{s}_Y}\Bigr) -
   \kappa\Bigl(\tfrac{Y_k}{s_Y}, \tfrac{Y_l}{s_Y}\Bigr)\Bigr\}. 
 \end{gather*}
 There exists an $L$ depending only on $\kappa$ such
 that $|r_1|$ and $|r_2|$ is bounded from above by
 \begin{equation*}
   L(m+n)\Bigl\{ \frac{\|\mathbf{X} - \hat{\mathbf{X}} \mathbf{W}\|_{2
     \rightarrow \infty}}{(n-1) \hat{s}_{X}} + \frac{|s_{X} -
     \hat{s}_{X}|}{(n-1) s_{X} \hat{s}_{X}} + \frac{\|\mathbf{Y} -
     \hat{\mathbf{Y}} \mathbf{W}\|_{2
     \rightarrow \infty}}{(m-1) \hat{s}_{Y}} + \frac{|s_{Y} -
     \hat{s}_{Y}|}{(m-1) s_{Y} \hat{s}_{Y}} \Bigr\}.
 \end{equation*}
 Lemma~\ref{lem:2} implies $|r_1 + r_2| \rightarrow 0$ almost
 surely. Now denote $\sigma_X = (\mathbb{E}[\|X\|^{2}])^{1/2}$ and $\sigma_Y =
  (\mathbb{E}[\|Y\|^{2}])^{1/2}$. Then $s_X$ and $s_Y$ are
  $\sqrt{n}$-consistent and $\sqrt{m}$-consistent estimators of
  $\sigma_X$ and $\sigma_Y$, respectively. When $F \upVdash G \circ c$,
 $\mu[F \circ \mathbf{T}_1 \circ \sigma_{X}^{-1}] = \mu[G \circ \mathbf{T}_2
 \circ \sigma_{Y}^{-1}]$. Denote by $\xi_{W}^{(X)}$ and
 $\xi_{W}^{(Y)}$ the quantities
 \begin{gather*}
   \xi_{W}^{(X)} = \sqrt{m+n}\Bigl( \sum_{i=1}^{n}
    \frac{\kappa(\mathbf{T}_1 X_i/\sigma_{X}, \cdot)
    - \mu[F \circ \mathbf{T}_1 \circ \sigma_{X}^{-1}]}{n}\Bigr), \\
  \xi_{W}^{(Y)} = \sqrt{m+n} \Bigl(\sum_{k=1}^{m} \frac{\kappa(\mathbf{T}_2 Y_k / \sigma_{Y}, \cdot) - \mu[G
    \circ \mathbf{T}_2 \circ \sigma_{Y}^{-1}]}{m}\Bigr). 
 \end{gather*}
 Then $\xi_{W} = \xi_W^{(X)} + \xi_{W}^{(Y)} + O(1)$ and hence
 $\xi_{W} - O(1)$ is once again a sum of independent mean zero
  random elements of $\mathcal{H}$. A Hilbert
  space concentration inequality similar to that of
  \citep[Theorem~3.5]{pinelis94:_optim_banac} yields that
  $\|\xi\|_{\mathcal{H}}$ is bounded in probability.

  We next bound $\|\xi_{W} - \hat{\xi}\|_{\mathcal{H}}$. We mimic the
  proof of Lemma~\ref{lem:emp_proc}, paying attention to the terms
  $\hat{s}_X$ and $s_X$. A Taylor expansion of $\kappa$ yields 
 \begin{equation*}
   \begin{split}
   \frac{1}{\sqrt{n}} \sum_{i=1}^{n} (\Phi(\tfrac{X_i}{s_{X}}) -
   \Phi(\tfrac{\mathbf{W}\hat{X}_i}{\hat{s}_X}))( \cdot) &= 
   \frac{1}{\sqrt{n}} \sum_{i=1}^{n} \partial \Phi\Bigl(\tfrac{X_i}{s_X}\Bigr)(\cdot)^{\top}
   \Bigl(\tfrac{\mathbf{W}\hat{X}_i}{\hat{s}_X} -
   \tfrac{X_i}{s_X}\Bigr) \\ & + \frac{1}{2 \sqrt{n}} \sum_{i=1}^{n}  
   \Bigl(\tfrac{\mathbf{W}\hat{X}_i}{\hat{s}_X} -
   \tfrac{X_i}{s_X}\Bigr)^{\top} \partial^{2}
   \Phi\Bigl(\tfrac{X_i^{*}}{s_X}\Bigr)(\cdot) 
   \Bigl(\tfrac{\mathbf{W}\hat{X}_i}{\hat{s}_X} -
   \tfrac{X_i}{s_X}\Bigr).
   \end{split}
\end{equation*}
The terms depending on $\partial^{2} \Phi$ in the above display is bounded as
\begin{equation*}
  \begin{split}
  \Bigl|\frac{1}{2 \sqrt{n}} \sum_{i=1}^{n} \Bigl(\tfrac{\mathbf{W}\hat{X}_i}{\hat{s}_X} -
   \tfrac{X_i}{s_X}\Bigr)^{\top} \partial^{2} \Phi
   \Bigl(\tfrac{X_i^{*}}{s_X}\Bigr)(\cdot) \Bigl(\tfrac{\mathbf{W}\hat{X}_i}{\hat{s}_X} -
   \tfrac{X_i}{s_X}\Bigr)\Bigr|  &\leq \frac{\sup_{Z \in \Omega} \|\partial^{2}
   \Phi(Z) \|}{2 \sqrt{n}} \sum_{i=1}^{n} \Bigl \|\tfrac{\mathbf{W}\hat{X}_i}{\hat{s}_X} -
   \tfrac{X_i}{s_X} \Bigr\|^{2} \\
& \leq \frac{\sup_{Z \in \Omega} \|\partial^{2}
   \Phi(Z) \| \|\hat{\mathbf{X}} \mathbf{W} -
 \mathbf{X}\|_{F}^{2}}{\sqrt{n} (\hat{s}_X)^{2}} 
  \end{split}
\end{equation*}
which converges to $0$ almost
surely. For the terms depending on $\partial \Phi$, we have
   \begin{equation*}
     \begin{split}
    \frac{1}{\sqrt{n}} \sum_{i=1}^{n} \partial \Phi\Bigl(\tfrac{X_i}{s_X}\Bigr)(\cdot)^{\top}
   \Bigl(\tfrac{\mathbf{W} \hat{X}_i}{\hat{s}_X} - \tfrac{X_i}{s_X}\Bigr) 
 & = \frac{1}{\sqrt{n}} \sum_{i=1}^{n} \partial
   \Phi\Bigl(\tfrac{X_i}{s_X}\Bigr)(\cdot)^{\top}
   \tfrac{\mathbf{W} \hat{X}_i - X_i}{\hat{s}_X} \\ &+ \frac{1}{\sqrt{n}}
   \sum_{i=1}^{n} \partial
   \Phi\Bigl(\tfrac{X_i}{s_X}\Bigr)(\cdot)^{\top} X_i \Bigl(\tfrac{\hat{s}_{X} -
     s_{X}}{\hat{s}_{X} s_{X}}\Bigr).
     \end{split}
   \end{equation*}
   The first sum on the right hand side of the above display can be
   bounded using a chaining argument identical to that in the proof of
   Lemma~\ref{lem:emp_proc} and an application of Slutsky's
   theorem (for $\hat{s}_{X} \rightarrow
   (\mathbb{E}[\|X\|^{2}])^{1/2}$ almost surely). For the second sum
   on the right hand side, we have
   \begin{equation*}
     \begin{split}
       \Bigl|\frac{1}{\sqrt{n}} \sum_{i=1}^{n} \partial
       \Phi\Bigl(\tfrac{X_i}{s_X}\Bigr)(\cdot)^{\top} X_i \Bigl(\frac{\hat{s}_{X}
         - s_{X}}{\hat{s}_{X} s_{X}}\Bigr)\Bigr| &= \Bigl|\frac{1}{\sqrt{n}}
       \sum_{i=1}^{n} \Phi\Bigl(\tfrac{X_i}{s_X}\Bigr)(\cdot)^{\top} X_i
       \Bigl(\frac{\hat{s}_{X}^{2} - s_{X}^{2}}{(\hat{s}_{X} +
         s_{X})\hat{s}_{X} s_{X}}\Bigr)\Bigr| \\
       & \leq \frac{\sup_{Z,Z' \in \Omega} |(\partial \Phi(Z))(Z')^{\top} Z|}
       {\sqrt{n}} \frac{|\|\hat{\mathbf{X}}\|_{F}^{2} -
         \|\mathbf{X}\|_{F}^{2}|}{(\hat{s}_{X} + s_{X})\hat{s}_{X}
         }.
     \end{split}
   \end{equation*}
   We note that $\|\hat{\mathbf{X}}\|_{F}^{2} =
   \|\mathbf{S}_{\mathbf{A}}^{1/2} \|_{F}^{2}$ and
   $\|\mathbf{X}\|_{F}^{2} = \|\mathbf{S}_{\mathbf{P}}^{1/2}
   \|_{F}^{2}$. Thus $|\|\hat{\mathbf{X}}\|_{F}^{2} -
   \|\mathbf{X}\|_{F}^{2}| \leq \sqrt{d} \|\mathbf{S}_{\mathbf{A}} -
   \mathbf{S}_{\mathbf{P}} \|_{F}$ by the Cauchy-Schwarz
   inequality. Lemma~3.2 in \citet{lyzinski13:_perfec} can then be
   applied to $\|\mathbf{S}_{\mathbf{A}} - \mathbf{S}_{\mathbf{P}}
   \|_{F}$ to show that $|\|\hat{\mathbf{X}}\|_{F}^{2} -
   \|\mathbf{X}\|_{F}^{2} |$ is of order $O(\log{n})$ with probability
   at least $1 - n^{-2}$; note that this bound for
   $\|\mathbf{S}_{\mathbf{A}} - \mathbf{S}_{\mathbf{P}} \|_{F}$ 
   is much stronger than
   that obtained from Weyl's inequality and a concentration bound for
   $\|\mathbf{A} - \mathbf{P}\|$ from
   \citet{oliveira2009concentration,rinaldo_2013,lu13:_spect}. Hence
   by the compactness of $\Omega$, smoothness of $\Phi$ and Slutsky's
   theorem, the second sum also converges to $0$ almost surely,
   thereby establishing Eq.~\eqref{eq:6}.

\subsection{Proof for the Projection Case \S\ref{sec:projection-case}}
\label{sec:projection_proof}
  The proof of this result is almost identical to that of
  Theorem~\ref{thm:mmd_unbiased_ase}. We note here the requisite
  modifications for the case when the null hypothesis of $F \circ \pi^{-1} \upVdash G \circ \pi^{-1}$ holds.
 Define $\xi_{W}, \hat{\xi} \in \mathcal{H}$ by 
 \begin{gather*}
  \xi_{W} = \frac{\sqrt{m+n}}{n} \sum_{i=1}^{n} \kappa(\mathbf{W}_1 \pi(X_i), \cdot) -
  \frac{\sqrt{m+n}}{m} \sum_{k=1}^{m} \kappa(\mathbf{W}_2 \pi(Y_k), \cdot), \\
  \hat{\xi} = \frac{\sqrt{m+n}}{n} \sum_{i=1}^{n} \kappa(\pi(\hat{X}_i), \cdot) -
  \frac{\sqrt{m+n}}{m} \sum_{k=1}^{m} \kappa(\pi(\hat{Y}_k), \cdot).
 \end{gather*}
 In addition, define $r_1 = r_{11} + r_{12}$ and $r_2 = r_{21} + r_{22}$ by
  \begin{gather*}
   r_{11} = \frac{m+n}{n(n-1)} \sum_{i=1}^{n} 
   \Bigl(\kappa(\pi(X_i), \pi(X_i)) - \kappa(\pi(\hat{X}_i),
   \pi(\hat{X}_i))\Bigr), \\ r_{12} = 
        \frac{m+n}{m(m-1)} \sum_{k=1}^{m} 
   \Bigl(\kappa(\pi(Y_k), \pi(Y_k)) - \kappa(\pi(\hat{Y}_k),\pi(\hat{Y}_k))\Bigr),  \\   
r_{21} = \frac{m+n}{n^{2}(n-1)} \sum_{i,j}
   \Bigl(\kappa(\pi(X_i), \pi(X_j)) -
   \kappa(\pi(\hat{X}_i),\pi(\hat{X}_j))\Bigr), \\ r_{22} = 
   \frac{m+n}{m^2(m-1)}
   \sum_{k,l}
   \Bigl(\kappa(\pi(Y_k), \pi(Y_l)) - \kappa(\pi(\hat{Y}_k), \pi(\hat{Y}_l))\Bigr).
 \end{gather*}
 Using the assumption that $\|Z\| \geq c_0$ $F$-almost everywhere for
 some constant $c_0 > 0$, both $|r_1|$ and $|r_2|$ can be
 bounded from above by
 \begin{equation*}
   L(m+n)\biggl\{ \frac{2\|\mathbf{X} - \hat{\mathbf{X}} \mathbf{W}\|_{2
     \rightarrow \infty}}{(n-1)c_0} + \frac{2 \|\mathbf{Y} -
   \hat{\mathbf{Y}} \mathbf{W}\|_{2
     \rightarrow \infty}}{(n-1) c_0}\biggr\}
 \end{equation*}
 for some constant $L$ depending only on $\kappa$. 
 To complete the proof, we adapt the argument in the proof of Lemma~\ref{lem:emp_proc}
 to the family of functions
 \begin{equation*}
   \mathcal{F} = \{f = (\partial ( \Phi \circ \pi)(\cdot))(Z) \colon Z
   \in \Omega \}
 \end{equation*}
 to show that $\| \xi_{W} - \hat{\xi} \|_{\mathcal{H}} \rightarrow 0$
 almost surely as $n \rightarrow \infty$.
\end{document}